\documentclass[11pt,a4paper,reqno]{amsart}
\usepackage[T1]{fontenc}
\usepackage{lmodern}
\usepackage[english]{babel}
\usepackage[margin=25mm,headheight=15pt]{geometry}
\usepackage{amsmath,amssymb,amsthm,mathtools,bm}
\usepackage{booktabs,array,longtable,enumitem,microtype}
\usepackage[numbers,sort&compress]{natbib}
\usepackage[colorlinks=true,linkcolor=blue,citecolor=blue,urlcolor=blue]{hyperref}
\usepackage{fancyhdr}
\usepackage{tikz}
\allowdisplaybreaks[2]
\setlist[enumerate]{itemsep=4pt,topsep=6pt}
\newcommand{\dd}{\,\mathrm d}
\newcommand{\vol}{\operatorname{Vol}}
\newcommand{\supp}{\operatorname{supp}}
\newcommand{\loc}{\mathrm{loc}}
\newcommand{\R}{\mathbb R}
\newcommand{\A}{\mathcal A}
\newcommand{\ip}[2]{\langle #1,#2\rangle}
\newcommand{\Deltam}{\Delta_m}
\newtheorem{theorem}{Theorem}[section]
\newtheorem{proposition}[theorem]{Proposition}
\newtheorem{lemma}[theorem]{Lemma}
\newtheorem{corollary}[theorem]{Corollary}
\theoremstyle{definition}
\newtheorem{definition}[theorem]{Definition}
\theoremstyle{remark}
\newtheorem{remark}[theorem]{Remark}
\newtheorem{example}[theorem]{Example}
\numberwithin{equation}{section}
\title[Sharp volume tests for gradient reactions]{Sharp Volume Tests and Entropy Thresholds for Quasilinear Elliptic Inequalities with Gradient Reactions}
\author{Yongheng Qin}
\address{School of Mathematical Sciences and LPMC, Nankai University,
	Tianjin 300071, P. R. China}
\email{1120220027@mail.nankai.edu.cn}
\author{Yuhua Sun}
\address{School of Mathematical Sciences and LPMC, Nankai University,
	Tianjin 300071, P. R. China}
\email{sunyuhua@nankai.edu.cn}
\author{Jie Xiao}
\address{Department of Mathematics and Statistics, Memorial University of Newfoundland,\\ St. John's, Newfoundland and Labrador, Canada A1C 5S7}
\email{jxiao@math.mun.ca}
\thanks{Y. Sun was funded by the National Natural Science Foundation of
	China (Grant No. 12371206) and the Fundamental Research Funds for the
	Central Universities, No. 050-63263078.}
\thanks{J. Xiao was supported in part by the Natural Sciences and Engineering Research Council of Canada (NSERC), Discovery Grant 202979.}
\makeatletter
\@ifundefined{subjclassname@2020}{\@namedef{subjclassname@2020}{\textup{2020} Mathematics Subject Classification}}{}
\makeatother
\hypersetup{pdftitle={Sharp Volume Tests and Entropy Thresholds for Quasilinear Elliptic Inequalities with Gradient Reactions},pdfauthor={Yongheng Qin, Yuhua Sun, and Jie Xiao},pdfkeywords={Liouville theorem, m-Laplacian, gradient nonlinearity, volume growth, weighted manifolds}}
\begin{document}
	\begin{abstract}
		We investigate positive weak solutions of
		$-\Delta_m u\ge u^p|\nabla u|^q$ on complete noncompact
		Riemannian manifolds without curvature, volume doubling, or Poincar\'e
		assumptions. For the superhomogeneous mixed reaction we derive a
		curvature-free integral volume test by combining two first-order
		estimates with a recurrence for the mass of a \emph{fixed} density.
		The same approach organizes the pure-gradient and capacity criteria
		and yields sharp finite hierarchies of iterated-logarithmic growth
		bounds, with counterexamples on smooth model manifolds. In the
		subhomogeneous region, vanishing lower growth of
		$\log V(R)/R^{m-q}$ forces the Liouville property. On the
		homogeneous line $p+q=m-1$, we obtain the optimal strict lower-volume-
		entropy threshold
		$(p+1)((m-1)/p)^{p/(p+1)}$ for $p>0$ (and $1$
		for $p=0$), together with examples at equality. The existing six-region classification is not claimed as new;
		the advances concern refined volume hypotheses and their model
		sharpness. The tests remain
		valid for smooth weighted measures and for locally bounded positive
		Sobolev solutions. Sharpness is understood within the stated
		model families; no characterization from volume alone on an
		individual manifold is asserted.
	\end{abstract}
	\maketitle
	\enlargethispage{3pt}
	
	\section{Introduction}
	\subsection{The problem and its geometric setting}
	Let $m>1$. We consider the inequality
	\begin{equation}\label{eq:main}
		-\Delta_m u\ge u^p|\nabla u|^q,\qquad
		\Delta_m u=\operatorname{div}(|\nabla u|^{m-2}\nabla u),
	\end{equation}
	on a smooth, connected, complete noncompact Riemannian manifold $M$
	without boundary. The parameters $p,q$ are real. For a fixed point $o\in M$,
	write $r(x)=d(o,x)$, $B_R=B(o,R)$, and $V(R)=\vol(B_R)$.
	Our question is how the growth of $V$ alone can force every positive
	solution to be constant, or exclude positive solutions altogether.
	No pointwise or integral bound on Ricci or sectional curvature is imposed.
	The results do not require volume doubling, a global Poincar\'e inequality,
	or bounds for $\Delta r$. No global boundedness or global integrability
	of the solution is assumed.
	
	The factor $|\nabla u|^q$ changes both the analytic estimates and the
	meaning of the conclusion. When $q>0$, positive constants solve
	\eqref{eq:main}; the appropriate Liouville property is constancy.
	When $q\le0$, the conclusion is nonexistence in the admissible class
	specified below. Negative gradient powers are interpreted through local
	integrability, without assigning them the value zero at critical points.
	This convention is essential for the singular parameter range.
	
	\subsection{Test functions, volume growth, and earlier results}
	Nonlinear test functions and capacity estimates form a general approach
	to nonexistence for differential inequalities; see Mitidieri--Pohozaev
	\cite{MP2001} and Serrin--Zou \cite{SZ2002}.
	The role of volume in potential theory is developed in
	Grigor'yan \cite{Grig1999}, Holopainen \cite{Hol1999}, and
	Rigoli--Setti \cite{RS2001}.
	For nonlinear inequalities on complete manifolds, the distinction between
	volume information and curvature information is substantial:
	first-order weak tests can be used across the cut locus, whereas
	pointwise differentiation of distance cutoffs generally needs additional
	geometric control.
	
	For $-\Delta u\ge u^p$, $p>1$,
	Grigor'yan--Sun \cite{GS2014} proved nonexistence under
	\begin{equation}\label{eq:GSvolume}
		V(R)\le C R^{2p/(p-1)}(\log R)^{1/(p-1)}.
	\end{equation}
	Sun \cite{Sun2015P} obtained the corresponding quasilinear result.
	Wang--Xiao \cite{WX2016} developed flexible model constructions for
	the sharpness of the power and first logarithmic exponent.
	Zhang \cite[Theorem 1.2]{Zhang2015} gave a different, curvature-free
	integral sufficient condition in the semilinear case:
	\[
	\liminf_{t\downarrow0}t^{p/(p-1)}
	\int_1^\infty\frac{V(R)}{R^{(3p-1)/(p-1)+t}}\dd R<\infty.
	\]
	Thus integral volume conditions already have a substantial history.
	Mastrolia--Monticelli--Punzo \cite{MMP2015} treat the interaction
	between geometry and a variable potential, while
	Grigor'yan--Sun--Verbitsky \cite{GSV2020} characterize superlinear
	inequalities through Green kernels and obtain volume formulations under
	additional analytic hypotheses.
	
	Earlier gradient-dependent inequalities on manifolds were studied by
	Sun \cite{Sun2015}. Bidaut-V\'eron--Garc\'ia-Huidobro--V\'eron
	\cite{BGV2019} develop the Euclidean product-reaction theory, including
	the distinction between solution estimates and supersolution Liouville ranges.
	The closest general predecessor is Sun--Xiao--Xu \cite{SXX2022},
	which studies exactly \eqref{eq:main} throughout the real $(p,q)$-plane.
	Its six-region decomposition identifies the distinct polynomial,
	logarithmic, and exponential growth scales. We retain that decomposition
	and compare the corresponding volume assumptions in Table~\ref{tab:comparison}.
	The capacity criterion in the high-gradient region was already noted
	in \cite{SXX2022}; it is included here for completeness.
	For a pure gradient source, Goffi \cite[Theorem 1.2]{Goffi2023}
	proves a curvature-free criterion involving the areas of geodesic spheres,
	for general divergence-form operators and without a one-sided bound on
	the solution. This overlaps the $p=0$ specialization of our integral
	criterion and should be kept distinct from the mixed reaction problem.
	The case $q=0$ of our first theorem also agrees with the elliptic
	integral criterion in  Gu--Hao--Huang--Sun \cite{GHHS2026}  using flow decomposition and Qin--Sun \cite{QS2026} using test function, where a larger Sobolev
	supersolution class and the parabolic Cauchy problem are studied.
	
	\subsection{Recent developments and the position of the present results}
	Recent work also studies the product reaction by methods that retain
	more differential information about the solution. Ma--Wu \cite{MW2026}
	obtain improved Euclidean Liouville ranges for the equation
	$-\Delta u=u^p|\nabla u|^q$ by differential identities.
	He--Hu--Wang \cite{HHW2026} use Nash--Moser iteration for logarithmic
	gradient estimates on manifolds with a Ricci lower bound.
	Lu--Zhu \cite{LZ2026} investigate general product gradient
	structures for quasilinear equations in Euclidean domains.
	These equation results, and their gradient estimates, address questions
	different from a curvature-free volume criterion for weak inequalities;
	their exponent ranges need not coincide with the supersolution ranges here.
	
	Discrete analogues provide a further comparison.
	Hao--Sun \cite{HS2023} study gradient reactions on weighted graphs,
	and Duong--Liu--Minh--Quyet--Sun \cite{DLMQS2027} extend this direction to the discrete $m$-Laplacian.
	The discrete chain rule requires additional care, so our continuous
	argument does not automatically prove the corresponding graph statements.
	These comparisons locate the present results within both the earlier
	volume-growth theory and developments available by October 2026.
	
	\subsection{What is new beyond the 2022 classification}
	The starting point of this paper is the complete six-region parameter
	classification of Sun--Xiao--Xu \cite{SXX2022}, not a reclassification
	of the $(p,q)$-plane.  The distinction is between a \emph{growth envelope}
	of the form $R^D(\log R)^B$ and a \emph{scale-integral criterion}
	that may hold with no pointwise polynomial--logarithmic upper bound.
	In $G_1$, the genuinely mixed case $p>0$ is addressed by a two-stage
	negative-power and ordinary-cutoff test, followed by a fixed-density
	annular recurrence.  The $p=0$ case reduces to a pure gradient
	inequality, and the $q=0$ specialization is also covered by
	 Gu--Hao--Huang--Sun \cite{GHHS2026} and Qin--Sun \cite{QS2026}; neither specialization is advertised here as
	an independent discovery.
	
	In $G_3$ the criterion can alternatively be obtained from Goffi's
	pure-gradient theory \cite{Goffi2023} after a concave transformation.
	In $G_2$ the capacity criterion was already available in
	\cite{SXX2022}.  We include these regions to make the mechanism and the
	iterated-logarithmic scale uniform, not to claim priority for their
	underlying Liouville theorems.  In $G_4$ and $G_6$ the assertion is the
	zero \emph{lower} $R^{m-q}$-growth condition; in $G_5$ it is the exact
	\emph{strict} entropy threshold, whose numerical value was already
	visible in the earlier model construction.  These are refinements of
	the volume hypotheses rather than a new parameter-range theorem.
	
	A further limitation is essential: the counterexamples constructed in
	this work prove sharpness of the stated growth scales \emph{among
		smooth model manifolds}. They do not turn the divergence criterion into
	a necessary and sufficient condition on each fixed geometry.
	
	\subsection{Contributions, generality, and sharpness}
	The principal analytic contribution is a fixed-density annular estimate for the mixed-gradient regime.
	For $p\ge0$, $p+q>m-1$, $q<m$, it yields
	\begin{equation}\label{eq:introintegral}
		\int^\infty
		\frac{R^{(mp+q)/(m-1)-1}}{V(R)^{(p+q-m+1)/(m-1)}}\dd R=\infty
		\quad\Longrightarrow\quad\text{the Liouville property}.
	\end{equation}
	For $p<0$ and $m-1<q<m$, a different fixed weight gives a criterion
	independent of $p$. Each criterion permits any fixed finite number of
	critical iterated logarithms. The same density is used at every radius:
	a negative-power test controls the energy, an ordinary test removes the
	auxiliary solution power, and a mass recurrence is summed across scales.
	This avoids a limiting argument for tails with radius-dependent weights.
	
	A separate contribution concerns faster growth.
	When $p+q<m-1$ and $q\le m-1$, keeping the full parameter dependence
	in Young's inequality gives the sufficient condition
	\[
	\liminf_{R\to\infty}\frac{\log V(R)}{R^{m-q}}=0.
	\]
	On the line $p+q=m-1$, $p\ge0$, a logarithmic transformation instead
	gives the exact strict entropy threshold stated in Theorem~\ref{thm:G5}.
	The latter value agrees with the model constant in
	\cite{SXX2022}; our nonexistence argument reaches that
	strict value and the endpoint is attained by examples.
	
	The geometric generality and the sharpness assertions have precise meanings.
	Our test-function proofs use first derivatives of compactly supported
	tests only. They therefore extend to arbitrary smooth positive densities,
	with weighted volume replacing Riemannian volume, without a
	Bakry--\'Emery curvature assumption; see Theorem~\ref{thm:weighted}.
	Theorem~\ref{thm:logsharp} constructs smooth complete models immediately
	beyond every finite iterated-logarithm threshold in the three integral
	regions. Hence these logarithmic exponents are sharp within those model
	families. This does not assert that convergence of the integral ensures
	existence on every prescribed manifold.
	In the subhomogeneous region the separating power $m-q$ is sharp from
	the two sides, but general critical growth $\exp(CR^{m-q})$ remains
	outside the conclusions proved here.
	
 Our paper is organized as follows:	Section~\ref{sec:results} states the results and compares their hypotheses.
	Sections~\ref{sec:prelim}--\ref{sec:entropy} give the testing estimates
	and nonexistence proofs. Section~\ref{sec:models} proves the sharpness
	statements by global models. Section~\ref{sec:weighted} establishes
	the weighted extension.
	
	\section{Main results}\label{sec:results}
	\subsection{Geometric setting, weak solutions, and parameter regions}
	Throughout, $M$ is a smooth, connected, geodesically complete,
	noncompact Riemannian manifold without boundary. Its dimension is
	denoted by $n$ and its volume measure by $\mu$; the model constructions
	use $n\ge2$. No curvature assumption is imposed on any manifold in a nonexistence theorem.
	Particular geometries are used only to construct examples.
	Let $\operatorname{Lip}_c(M)$ denote the space of compactly supported
	Lipschitz functions.
	
	\begin{definition}\label{def:solution}
		Let $m>1$ and $p,q\in\R$. A function $u$ is a positive admissible weak
		solution of \eqref{eq:main} if $u\in C^1(M)$, $u>0$,
		$f=u^p|\nabla u|^q\in L^1_{\loc}(M)$, and
		\begin{equation}\label{eq:weak}
			\int_M|\nabla u|^{m-2}\ip{\nabla u}{\nabla\psi}\dd\mu
			\ge \int_M f\psi\dd\mu
		\end{equation}
		for every nonnegative $\psi\in\operatorname{Lip}_c(M)$.
		When $q=0$, we use the convention $|\nabla u|^0=1$.
		When $q<0$, the negative power at critical points is understood in the
		improper-integral sense and is not redefined to be zero.
		In particular, local integrability requires $|\nabla u|>0$ almost
		everywhere (because $u$ is locally bounded above and away from zero).
	\end{definition}
	
	All proofs are first established in the positive $C^1$ class above.
	Its local positivity makes every negative-power test admissible.
	The Sobolev formulations in earlier work should be distinguished from
	this class: using a larger collection of test functions does not by
	itself imply $C^1$ regularity of the solution.
	The extension in Proposition~\ref{prop:sobolev} applies to positive
	Sobolev supersolutions with two-sided local essential bounds.
	No assertion for arbitrary singular or unbounded Sobolev solutions,
	or for solutions that vanish, is made here.
	
	\begin{definition}
		In this paper, the Liouville property means that every positive
		admissible weak solution is constant if $q>0$, and that no positive
		admissible weak solution exists if $q\le0$.
	\end{definition}
	
	We use the following six disjoint regions:
	\begin{equation}\label{eq:regions}
		\begin{aligned}
			G_1&=\{(p,q):p\ge0,\ m-1-p<q<m\},\\
			G_2&=\{(p,q):q\ge m\},\\
			G_3&=\{(p,q):p<0,\ m-1<q<m\},\\
			G_4&=\{(p,q):p<0,\ q=m-1\},\\
			G_5&=\{(p,q):p=m-1-q,\ q\le m-1\},\\
			G_6&=\{(p,q):p<m-1-q,\ q<m-1\}.
		\end{aligned}
	\end{equation}
	They cover the entire $(p,q)$-plane.

	Figure~\ref{fig:pq-regions} depicts this partition in the $(p,q)$-plane.
	The drawing uses the shifted vertical coordinate $q-(m-1)$,
	so it applies without choosing a numerical value of $m>1$.
	Solid boundary rays are assigned to $G_2$, $G_4$, or $G_5$ as
	specified in \eqref{eq:regions}; the other regions are open
	along the separating rays. In particular, the sloping boundary
	$G_5$ is restricted to $p\ge0$ and is not extended into $p<0$.
	
	\begin{figure}[htbp]
		\centering
		\begin{tikzpicture}[x=0.98cm,y=0.98cm,>=stealth,
			every node/.style={font=\small}]
			% Plot coordinates: X=p+3 and Y=q-(m-1)+2.5.
			% The window represents -3<=p<=5, -2.5<=q-(m-1)<=2.5.
			\begin{scope}
				\clip (0,0) rectangle (8,5);
				\fill[blue!9] (0,3.5) rectangle (8,5); % G2
				\fill[orange!16] (0,2.5) rectangle (3,3.5); % G3
				\fill[green!11] (0,0) rectangle (8,2.5); % G6 baseline
				\fill[red!13] (3,2.5)--(3,3.5)--(8,3.5)--(8,0)--(5.5,0)--cycle; % G1
				% Boundaries of the classification
				\draw[very thick,blue!65!black] (0,3.5)--(8,3.5);
				\draw[very thick,orange!80!black] (0,2.5)--(3,2.5);
				\draw[very thick,purple!70!black] (3,2.5)--(5.5,0);
			\end{scope}
			% Coordinate axes: solid black, with visually offset labels
			\draw[->,thick] (0,0)--(8.45,0) node[right] {$p$};
			\draw[->,thick] (3,0)--(3,5.35) node[above] {$q$};
			\draw (2.92,2.5)--(3.08,2.5);
			\node[above left] at (3,2.5) {$m-1$};
			\node[left] at (3,3.5) {$m$};
			\node[below right] at (3,0) {$0$};
			\node[align=center,blue!65!black] at (5.5,4.25) {$G_2$};
			\node[align=center,orange!80!black] at (1.45,3.05) {$G_3$};
			\node[align=center,red!75!black] at (5.5,3.05) {$G_1$};
			\node[align=center,teal!65!black] at (1.35,1.45) {$G_6$};
			\node[align=center,teal!65!black] at (4.1,.75) {$G_6$};
			\node[text=purple!75!black,fill=white,inner sep=1pt] at (4.15,1.8) {$G_5$};
			\node[above,orange!85!black,fill=white,inner sep=1pt] at (1.25,2.5) {$G_4$};
			\node[font=\footnotesize,above] at (6.75,.22) {$p+q=m-1$};
			\draw[->,thin,purple!70!black] (6.05,.33)--(5.2,.36);
			\node[font=\footnotesize,above] at (7.0,3.5) {$q=m$};
		\end{tikzpicture}
		\caption{The six parameter regions in the $(p,q)$-plane, with the
			vertical axis shown at $p=0$. The horizontal separating level is
			$q=m-1$, and the descending ray represents $p+q=m-1$ for $p\ge0$.
			The region $G_2$ includes $q=m$, $G_4$ is the left horizontal ray,
			and $G_5$ is the right descending ray. Colors serve only to
			identify the regions; no geometric information is encoded by them.}
		\label{fig:pq-regions}
	\end{figure}
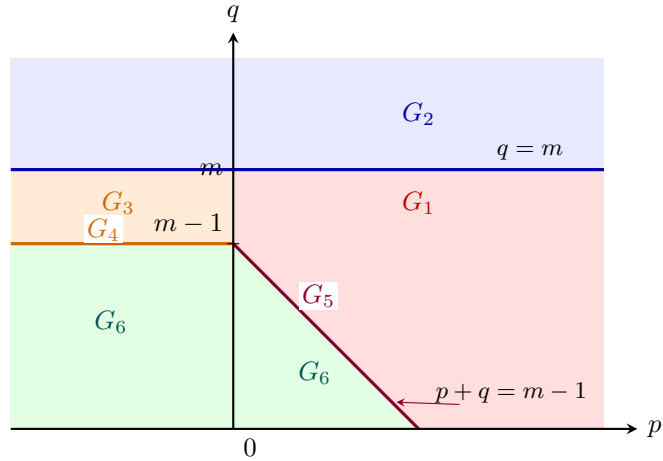
	Every unspecified lower limit in an integral below is a sufficiently
	large positive number; changing such a finite lower limit does not
	affect divergence. Set
	\begin{equation}\label{eq:DB}
		\begin{array}{c|cc}
			i&D_i&B_i\\ \hline
			1&\displaystyle\frac{mp+q}{p+q-m+1}
			&\displaystyle\frac{m-1}{p+q-m+1}\\[6pt]
			2&m&m-1\\[3pt]
			3&\displaystyle\frac q{q-m+1}
			&\displaystyle\frac{m-1}{q-m+1}.
		\end{array}
	\end{equation}
	These quantities are positive in their respective regions. The index $i=2$ is a capacity-theoretic benchmark, not a new result; the $i=3$ criterion is also reducible to a known pure-gradient theorem by a concave transformation.
	
	\subsection{Integral criteria and iterated logarithms}
	\begin{theorem}[Integral criterion in the superhomogeneous region]\label{thm:G1}
		Let $(p,q)\in G_1$. If
		\begin{equation}\label{eq:G1int}
			\int^\infty
			\frac{R^{(mp+q)/(m-1)-1}}
			{V(R)^{(p+q-m+1)/(m-1)}}\dd R=\infty,
		\end{equation}
		then \eqref{eq:main} has the Liouville property.
		More precisely, every positive admissible weak solution would satisfy
		$u^p|\nabla u|^q=0$ almost everywhere.
	\end{theorem}
	
	\begin{theorem}[Integral criterion for a negative solution exponent]\label{thm:G3}
		Let $(p,q)\in G_3$. If
		\begin{equation}\label{eq:G3int}
			\int^\infty
			\frac{R^{q/(m-1)-1}}{V(R)^{(q-m+1)/(m-1)}}\dd R=\infty,
		\end{equation}
		then every positive admissible weak solution is constant.
		This condition is independent of the value of $p$.
	\end{theorem}
	The conclusion can also be recovered from Goffi~\cite{Goffi2023} after a
	concave change of variable; see Remark~\ref{rem:Goffi}. The proof below
	records its place in the same fixed-density method as Theorem~\ref{thm:G1}.
	
	\begin{proposition}[Capacity criterion]\label{prop:G2}
		If
		\begin{equation}\label{eq:G2int}
			\int^\infty\left(\frac R{V(R)}\right)^{1/(m-1)}\dd R=\infty,
		\end{equation}
		then every positive $C^1$ weakly $m$-superharmonic function is constant.
		In particular, \eqref{eq:main} has the Liouville property for
		$(p,q)\in G_2$.
	\end{proposition}
	The volume condition in Proposition \ref{prop:G2} was already given
	in \cite{SXX2022} and is not claimed as an improvement
	over that work. We include its proof in Section~\ref{sec:integral} to make explicit
	that the entire approach uses only first derivative estimates for
	cutoffs.
	
	The preceding three conditions can be written in the unified form
	\begin{equation}\label{eq:unified}
		\mathcal I_i(V):=
		\int^\infty \left(\frac{R^{D_i}}{V(R)}\right)^{1/B_i}
		\frac{\dd R}{R}=\infty,\qquad i=1,2,3.
	\end{equation}
	For instance, when $i=1$ we have $D_1/B_1=(mp+q)/(m-1)$ and
	$1/B_1=(p+q-m+1)/(m-1)$, which recovers \eqref{eq:G1int}.
	
	\begin{corollary}[Any finite number of iterated logarithms]\label{cor:logs}
		Let $L_1(R)=\log R$ and $L_{j+1}(R)=\log L_j(R)$.
		For $i=1,2,3$, fix any positive integer $N$. If, for all sufficiently
		large $R$,
		\begin{equation}\label{eq:logs}
			V(R)\le C R^{D_i}
			\prod_{j=1}^{N-1}L_j(R)^{B_i}\,L_N(R)^\beta,
			\qquad \beta\le B_i,
		\end{equation}
		then the Liouville property holds in the corresponding parameter
		region.
	\end{corollary}
	
	For $N=1$, this recovers the previously known critical first-logarithm
	condition. For $N=2$, it permits the additional factor
	$(\log\log R)^{B_i}$; for $N=3$, it also permits
	$(\log\log\log R)^{B_i}$. In each case $N$ is fixed and finite.
	No infinite product or infinitely iterated cutoff construction is
	involved.

	\begin{theorem}[Sharpness at every finite logarithmic level]\label{thm:logsharp}
		Fix $(p,q)\in G_i$, where $i\in\{1,2,3\}$, and let $D_i,B_i$ be
		as in \eqref{eq:DB}. For every integer $N\ge1$, every $b>B_i$, and
		every dimension $n\ge2$, there exists a smooth complete model on
		$\mathbb R^n$ with
		\begin{equation}\label{eq:logsharpvolume}
			V(R)\asymp R^{D_i}
			\prod_{j=1}^{N-1}L_j(R)^{B_i}L_N(R)^b
		\end{equation}
		on which \eqref{eq:main} has a positive nonconstant admissible
		$C^1$ weak solution.
	\end{theorem}
	
	Corollary~\ref{cor:logs} and Theorem~\ref{thm:logsharp} determine the
	exact threshold of the last logarithmic exponent in each family
	\eqref{eq:logsharpvolume}: the Liouville property holds on every manifold satisfying the displayed
	upper bound for $b\le B_i$, while suitable model manifolds with
	that asymptotic volume order admit counterexamples for $b>B_i$.
	The topological dimension $n$ is arbitrary; it is the asymptotic
	volume exponent $D_i$ that controls this comparison.
	The examples do not prescribe the geometry of an arbitrary manifold
	whose integral happens to converge.
	
	\begin{remark}\label{rem:necessary}
		By contraposition, the corresponding integral must converge if a
		positive nonconstant solution exists. For $q\le0$, existence of any
		positive admissible solution also requires convergence.
		The model sharpness in Theorem~\ref{thm:logsharp} does not prove
		that convergence implies existence on an arbitrary given manifold. Thus the general integral criterion should not be
		described as an unconditional necessary and sufficient condition.
	\end{remark}
	
	\subsection{Zero lower volume entropy and subhomogeneous growth}
	For $s>0$, define the lower growth rate
	\begin{equation}\label{eq:entropy}
		h_s(M)=\liminf_{R\to\infty}\frac{\log V(R)}{R^s}.
	\end{equation}
	Since $V$ is nondecreasing and has a fixed positive lower bound for
	large radii, $h_s(M)\ge0$. The usual lower volume entropy is denoted
	by $h_1(M)$.
	
	\begin{theorem}[The subhomogeneous region]\label{thm:G46}
		Suppose $p+q<m-1$ and $q\le m-1$. If
		\begin{equation}\label{eq:G46}
			h_{m-q}(M)=0,
		\end{equation}
		then \eqref{eq:main} has the Liouville property. In particular:
		\begin{enumerate}
			\item for $(p,q)\in G_4$, the sufficient condition is $h_1(M)=0$;
			\item for $(p,q)\in G_6$, the sufficient condition is
			$\liminf\limits_{R\to\infty}R^{q-m}\log V(R)=0$.
		\end{enumerate}
	\end{theorem}
	
	Condition \eqref{eq:G46} requires only a suitable sequence of radii,
	rather than a uniform subcritical estimate for all large radii.
	In $G_4$, it includes polynomial growth of any fixed order, as well as
	\[
	V(R)\le C\exp(R^\alpha)\ (0<\alpha<1),
	\qquad
	V(R)\le C\exp\!\left(\frac R{\log R}\right).
	\]
	In $G_6$, let $k=m-q>1$. Each of the following is sufficient:
	\begin{equation}\label{eq:G6examples}
		\begin{aligned}
			V(R)&\le C\exp(cR^\alpha),&&0<\alpha<k,\\
			V(R)&\le C\exp(cR\log R),&&c>0,\\
			V(R)&\le C\exp\!\left(\frac{R^k}{\log R}\right).
		\end{aligned}
	\end{equation}
	In particular, the middle condition no longer restricts $c$ by a
	prescribed constant depending on $p,q,m$.
	
	\subsection{The exact entropy threshold on the homogeneous line}
	When $(p,q)\in G_5$, we have $p\ge0$ and $p+q=m-1$.
	Define
	\begin{equation}\label{eq:cstar}
		c_*(m,p)=
		\begin{cases}
			(p+1)\left(\dfrac{m-1}{p}\right)^{p/(p+1)},&p>0,\\[5pt]
			1,&p=0.
		\end{cases}
	\end{equation}
	
	\begin{theorem}[Strict entropy threshold]\label{thm:G5}
		Let $(p,q)\in G_5$. If
		\begin{equation}\label{eq:G5}
			h_1(M)<c_*(m,p),
		\end{equation}
		then \eqref{eq:main} has the Liouville property.
		For every $m>1,p\ge0$, there exists a complete noncompact model with
		volume entropy $c_*(m,p)$ on which \eqref{eq:main} admits a positive
		nonconstant admissible weak solution. Consequently, within the class
		of all complete manifolds, \eqref{eq:G5} cannot be replaced by
		$h_1(M)\le c_*(m,p)$.
	\end{theorem}
	
	The constant \eqref{eq:cstar} agrees with $H(p)$ in
	\cite{SXX2022}. The improvement here is to raise the
	nonexistence threshold to the same strict value.
	For example, $c_*=2$ when $m=2,p=1,q=0$, and $c_*=1$ when
	$m=2,p=0,q=1$. The earlier sufficient exponential coefficient is
	$1/(2e)$ in both examples.
	
	\begin{proposition}[A constant source coefficient]\label{prop:coefficient}
		Replace \eqref{eq:main} by
		\begin{equation}\label{eq:coefficient}
			-\Deltam u\ge \lambda u^p|\nabla u|^q,\qquad \lambda>0.
		\end{equation}
		The volume conditions in Theorems \ref{thm:G1}, \ref{thm:G3},
		\ref{thm:G46}, and Proposition \ref{prop:G2} remain unchanged.
		In $G_5$, the threshold becomes
		\begin{equation}\label{eq:cstar-lambda}
			c_{*,\lambda}(m,p)=
			\begin{cases}
				(p+1)\left(\dfrac{m-1}{p}\right)^{p/(p+1)}
				\lambda^{1/(p+1)},&p>0,\\[5pt]
				\lambda,&p=0.
			\end{cases}
		\end{equation}
	\end{proposition}
	
	\subsection{Comparison with earlier conditions and the range of the models}
	Table \ref{tab:comparison} compares the assumptions in
	\cite{SXX2022} with the conditions used here. The pure gradient and capacity
	specializations have the prior interpretations described above.
	The notation $\lesssim$ denotes an upper bound up to a positive
	constant, always for sufficiently large radii.
	\begin{table}[htbp]
		\centering\small
		\caption{Comparison of volume conditions in the six regions.}\label{tab:comparison}
		\begin{tabular}{>{\raggedright\arraybackslash}p{1.2cm}
				>{\raggedright\arraybackslash}p{5.2cm}
				>{\raggedright\arraybackslash}p{6.75cm}}
			\toprule
			Region&Earlier condition \cite{SXX2022}&Condition used here\\
			\midrule
			$G_1$&$V\lesssim R^{D_1}(\log R)^{B_1}$
			&$\mathcal I_1(V)=\infty$, Theorem \ref{thm:G1}\\[5pt]
			$G_2$&$V\lesssim R^m(\log R)^{m-1}$
			&$\mathcal I_2(V)=\infty$; already noted in the original paper\\[5pt]
			$G_3$&$V\lesssim R^{D_3}(\log R)^{B_3}$
			&$\mathcal I_3(V)=\infty$, Theorem \ref{thm:G3}\\[5pt]
			$G_4$&$V\lesssim R^\alpha$, fixed $\alpha>0$
			&$h_1(M)=0$, Theorem \ref{thm:G46}\\[5pt]
			$G_5$&$V\lesssim e^{\kappa R}$,\newline
			$\kappa<\min\{m-1,1\}/(2e)$
			&$h_1(M)<c_*(m,p)$, Theorem \ref{thm:G5}\\[5pt]
			$G_6$&$V\lesssim e^{\kappa R\log R}$,\newline
			$\kappa<(m-1-q)/(m-1-p-q)$
			&$h_{m-q}(M)=0$, Theorem \ref{thm:G46}\\
			\bottomrule
		\end{tabular}
	\end{table}
	
	\begin{proposition}[Model growth scales in the subhomogeneous region]\label{prop:models}
		Let $(p,q)\in G_6$, and write $d_0=m-1-q>0$ and $k=m-q=d_0+1$.
		\begin{enumerate}
			\item For every $\ell>k$ and $\lambda>0$, there exists a smooth
			complete model manifold with $\log V(R)\sim\lambda R^\ell$ on which
			\eqref{eq:main} admits a positive nonconstant admissible weak solution.
			\item For every $\beta>d_0$ and $\lambda>0$, such a model also exists
			with $\log V(R)\sim\lambda R^k(\log R)^\beta$.
		\end{enumerate}
		For $(p,q)\in G_4$, any prescribed positive volume entropy $H>0$
		can coexist with a positive nonconstant solution on a suitable
		hyperbolic space.
	\end{proposition}
	
	Theorem \ref{thm:G46} and Proposition \ref{prop:models} identify the
	separating power scale $m-q$ in $G_6$ from opposite sides:
	exponential growth with a power strictly below this value satisfies
	the nonexistence condition, whereas model examples exist for every
	strictly larger power. These results do not settle general critical
	growth of the form $\exp(CR^{m-q})$, nor do they classify all
	logarithmic corrections at that power. Likewise, for the integral
	criteria, \eqref{eq:logs} is asserted only as a sufficient condition;
	convergence of a comparison integral is not a substitute for an
	existence construction.
	
	\section{First-order tests and annular summation}\label{sec:prelim}
	\subsection{First-order cutoffs and admissibility of the test functions}
	By the Hopf--Rinow theorem, closed bounded balls in $M$ are compact.
	The distance function $r$ is $1$-Lipschitz, and hence its weak gradient satisfies
	\begin{equation}\label{eq:distancegrad}
		|\nabla r|\le1\quad\text{almost everywhere}.
	\end{equation}
	For a Lipschitz function $\chi:[0,\infty)\to[0,1]$,
	the chain rule gives
	$|\nabla(\chi(r))|\le|\chi'(r)|$ almost everywhere.
	We never differentiate $\chi(r)$ a second time,
	so no classical second derivative needs to be defined on the cut locus.
	
	For the integral criteria, we use the cutoff
	\begin{equation}\label{eq:cutoff-annulus}
		\phi_R(x)=
		\begin{cases}
			1,&r(x)\le2R,\\
			(4R-r(x))/(2R),&2R<r(x)<4R,\\
			0,&r(x)\ge4R.
		\end{cases}
	\end{equation}
	Writing $\A_R=\{x:2R<r(x)<4R\}$, we have
	\begin{equation}\label{eq:cutoffbounds}
		0\le\phi_R\le1,\quad
		\supp\nabla\phi_R\subset\overline{\A_R},\quad
		|\nabla\phi_R|\le\frac1{2R}.
	\end{equation}
	Geodesic spheres have zero volume measure; including the boundary of the annulus
	in the notation therefore does not affect the integrals.
	The junctions of the piecewise linear cutoff produce no additional terms,
	since only weak first derivatives are used throughout the paper.
	
	If $K$ is a compact set containing the support of a test function, then
	\[
	0<\min_Ku\le\max_Ku<\infty,\qquad
	\sup_K|\nabla u|<\infty.
	\]
	Consequently, for a fixed real number $a$ and $b\ge1$,
	$u^{-a}\phi^b$ is a compactly supported Lipschitz function.
	Even when $a,b$ are subsequently chosen as functions of the radius,
	the test associated with each fixed radius remains admissible.
	What then requires separate verification is the dependence of the constants
	on these parameters, rather than the local admissibility of the test itself.
	
	If the distributional inequality is initially defined using smooth compactly
	supported tests, nonnegative smoothing in finitely many coordinate charts,
	together with a partition of unity, produces a uniformly bounded sequence of
	smooth tests supported in a fixed compact set.
	This sequence converges uniformly to the prescribed Lipschitz test,
	and its gradients converge in $L^m$.
	The flux $|\nabla u|^{m-2}\nabla u$ belongs to
	$L^{m/(m-1)}$ on that compact set, so H\"older's inequality permits passage to
	the limit on the left-hand side.
	On the right-hand side, the limit follows from $f\in L^1(K)$ and uniform
	convergence. Thus the use of Lipschitz tests in \eqref{eq:weak}
	imposes no substantive additional restriction.
	
	The basic identity used repeatedly below comes from the test
	$\psi=u^{-a}\phi^b$. Since
	\[
	\nabla(u^{-a}\phi^b)
	=-a u^{-a-1}\phi^b\nabla u
	+b u^{-a}\phi^{b-1}\nabla\phi,
	\]
	the weak inequality \eqref{eq:weak} yields
	\begin{equation}\label{eq:basicnegative}
		\int_M u^{p-a}|\nabla u|^q\phi^b\dd\mu
		+a\int_M u^{-a-1}|\nabla u|^m\phi^b\dd\mu
		\le b\int_{\supp\nabla\phi}
		u^{-a}|\nabla u|^{m-1}\phi^{b-1}|\nabla\phi|\dd\mu .
	\end{equation}
	For $a>0$, both terms on the left are nonnegative.
	Testing with the ordinary cutoff gives
	\begin{equation}\label{eq:basicordinary}
		\int_M u^p|\nabla u|^q\phi^b\dd\mu
		\le b\int_{\supp\nabla\phi}
		|\nabla u|^{m-1}\phi^{b-1}|\nabla\phi|\dd\mu .
	\end{equation}
	The right-hand sides of both inequalities involve only first derivatives
	of the cutoff. In what follows, the measure $\dd\mu$ will be omitted
	from integrals whenever there is no ambiguity.
	
	\subsection{An integral criterion from the recurrence of annular masses}
	We first isolate the argument common to the two integral criteria below.
	The density in this lemma must be fixed and cannot depend on $R$.
	
	\begin{lemma}[Summation of annular masses]\label{lem:annular}
		Let $F\ge0$ and $F\in L^1_{\loc}(M)$.
		For the cutoffs in \eqref{eq:cutoff-annulus},
		suppose that there exist fixed constants $b\ge1$, $Q>0$, $\gamma>0$, and $C>0$
		such that, for all sufficiently large $R$,
		\begin{equation}\label{eq:abstract-local}
			I_R:=\int_M F\phi_R^b
			\le C R^{-Q/(1+\gamma)}
			V(4R)^{\gamma/(1+\gamma)}
			\left(\int_{\A_R}F\phi_R^b\right)^{1/(1+\gamma)}.
		\end{equation}
		If
		\begin{equation}\label{eq:abstract-integral}
			\int^\infty\frac{R^{Q-1}}{V(R)^\gamma}\dd R=\infty,
		\end{equation}
		then $F=0$ almost everywhere.
	\end{lemma}
	
	\begin{proof}
		To retain the information carried by the annulus, set
		\[
		J_R=\int_{\A_R}F\phi_R^b,\qquad
		D=\frac Q\gamma,\qquad g_R=\frac{V(4R)}{R^D},\qquad
		c=C^{-(1+\gamma)}.
		\]
		Raising \eqref{eq:abstract-local} to the power $1+\gamma$ and rearranging gives
		\begin{equation}\label{eq:masslower}
			J_R\ge c\,g_R^{-\gamma}I_R^{1+\gamma}.
		\end{equation}
		On the other hand, $J_R\le I_R$.
		If $I_R>0$, division by $I_R$ yields
		\begin{equation}\label{eq:massupper}
			I_R\le c^{-1/\gamma}g_R.
		\end{equation}
		These estimates provide, respectively, a minimum amount of mass in the annulus
		and an upper bound for the mass in the inner ball.
		Both estimates must be retained.
		
		Suppose, to the contrary, that $F\not\equiv0$.
		Choose $R_*>0$ sufficiently large that
		$\int_{B_{R_*}}F>0$ and that the local estimate holds for $R\ge R_*$.
		Set
		\[
		R_j=8^jR_*,\qquad
		M_j=\int_{B_{R_j}}F,\qquad
		g_j=\frac{V(4R_j)}{R_j^D}\quad(j=0,1,\ldots).
		\]
		Since $\phi_{R_j}=1$ on $B_{2R_j}$,
		we have $M_j\le I_{R_j}$.
		Moreover,
		\[
		\A_{R_j}\subset B_{R_{j+1}}\setminus B_{R_j},
		\qquad 0\le\phi_{R_j}^b\le1.
		\]
		Thus \eqref{eq:masslower}--\eqref{eq:massupper} imply
		\begin{equation}\label{eq:recurrence}
			0<M_j\le c^{-1/\gamma}g_j,\qquad
			M_{j+1}-M_j\ge c\,g_j^{-\gamma}M_j^{1+\gamma}.
		\end{equation}
		Here $M_j$ is the mass of a fixed density without a cutoff weight.
		Consequently, the comparisons between different radii concern the same quantity.
		
		Define
		\[
		z_j=c\left(\frac{M_j}{g_j}\right)^\gamma.
		\]
		The first inequality gives $0<z_j\le1$,
		while the second gives $M_{j+1}\ge M_j(1+z_j)$.
		Taking negative $\gamma$-th powers, we obtain
		\begin{align}
			M_j^{-\gamma}-M_{j+1}^{-\gamma}
			&\ge M_j^{-\gamma}\bigl[1-(1+z_j)^{-\gamma}\bigr]\notag\\
			&=M_j^{-\gamma}\int_0^{z_j}\gamma(1+t)^{-\gamma-1}\dd t\notag\\
			&\ge\gamma2^{-\gamma-1}M_j^{-\gamma}z_j\notag\\
			&=\gamma2^{-\gamma-1}c\,g_j^{-\gamma}.
			\label{eq:telescoping}
		\end{align}
		Summing from $j=0$ to $N$, the left-hand side telescopes:
		\[
		\gamma2^{-\gamma-1}c\sum_{j=0}^N g_j^{-\gamma}
		\le M_0^{-\gamma}-M_{N+1}^{-\gamma}
		\le M_0^{-\gamma}.
		\]
		Letting $N\to\infty$, we find
		\begin{equation}\label{eq:seriesfinite}
			\sum_{j=0}^\infty \frac{R_j^Q}{V(4R_j)^\gamma}<\infty.
		\end{equation}
		
		We now relate the series to the integral using only the monotonicity of volume.
		For $4R_j\le R\le32R_j$, we have $V(R)\ge V(4R_j)$, and hence
		\begin{align}
			\int_{4R_j}^{32R_j}\frac{R^{Q-1}}{V(R)^\gamma}\dd R
			&\le \frac1{V(4R_j)^\gamma}
			\int_{4R_j}^{32R_j}R^{Q-1}\dd R\notag\\
			&=\frac{32^Q-4^Q}{Q}\,
			\frac{R_j^Q}{V(4R_j)^\gamma}.
			\label{eq:series-integral}
		\end{align}
		Since $4R_{j+1}=32R_j$, these intervals meet consecutively
		and cover $[4R_*,\infty)$.
		By \eqref{eq:seriesfinite}, the tail of the integral in
		\eqref{eq:abstract-integral} is finite,
		contrary to the hypothesis. Therefore $F=0$ almost everywhere.
	\end{proof}
	
	\begin{remark}
		The proof uses neither $V(8R)\le C V(R)$
		nor the sphere area $V'(R)$.
		The sole purpose of choosing a scale ratio of eight is to place each testing
		annulus between the two balls associated with successive masses;
		it does not prescribe a volume growth law.
		The lemma converts a local first-order testing estimate into a global integral
		criterion and is the common source of the iterated logarithmic conclusions
		in this paper.
	\end{remark}
	
	\section{Proofs of the integral volume criteria}\label{sec:integral}
	\subsection{The superhomogeneous region}
	In this subsection we prove Theorem \ref{thm:G1}.
	Write
	\[
	\delta=p+q-m+1>0,\qquad z=|\nabla u|,\qquad f=u^p z^q.
	\]
	Assume first that $p>0$.
	Choose $a>0$ sufficiently small but fixed; for example, require
	\begin{equation}\label{eq:choosea}
		a<\min\left\{\frac{m-1}{2},
		\frac{\delta}{2(1+|q-m+1|)},
		\frac{p(m-1)}{2(1+|q|)}\right\}.
	\end{equation}
	This choice of $a$ is independent of $R$ throughout the proof.
	
	\paragraph{Step 1: choosing two sets of H\"older weights.}
	Set
	\begin{equation}\label{eq:G1weights}
		\begin{aligned}
			K&=mp+q(a+1),\\
			\alpha&=\frac{p(m-1)+aq}{K},&
			\beta&=\frac{m-1-a}{K},&
			\lambda_0&=\frac{\delta+a}{K},\\
			\theta&=\frac{p(m-1)}{K},&
			\eta&=\frac{(a+1)(m-1)}{K},&
			\kappa_0&=\frac{\delta+a(q-m+1)}{K}.
		\end{aligned}
	\end{equation}
	To account for a possibly negative $q$ when determining the sign of the
	denominator, we first check the numerators directly.
	Condition \eqref{eq:choosea} gives
	\[
	p(m-1)+aq>\frac{p(m-1)}2>0,\quad
	m-1-a>\frac{m-1}2>0,\quad \delta+a>0.
	\]
	The sum of these three numerators is
	\[
	p(m-1)+aq+(m-1-a)+(\delta+a)
	=mp+q(a+1)=K,
	\]
	so $K>0$.
	Likewise, $\delta+a(q-m+1)>\delta/2>0$.
	All six weights are therefore strictly positive, and
	\begin{equation}\label{eq:G1sum}
		\alpha+\beta+\lambda_0=1,\qquad
		\theta+\eta+\kappa_0=1,\qquad
		1-\alpha=\frac{p+q}{K}>0.
	\end{equation}
	The positivity in the last identity follows from $p+q>m-1>0$.
	Consequently, the two sets of H\"older exponents
	\[
	\left(\frac1\alpha,\frac1\beta,\frac1{\lambda_0}\right),
	\qquad
	\left(\frac1\theta,\frac1\eta,\frac1{\kappa_0}\right)
	\]
	consist entirely of numbers strictly greater than $1$,
	and the reciprocals in each set sum to $1$.
	Fix
	\begin{equation}\label{eq:G1b}
		b\ge\max\{1/\lambda_0,1/\kappa_0\}.
	\end{equation}
	
	\paragraph{Step 2: controlling the energy by a negative-power test.}
	For the moment, let $\phi$ be any cutoff satisfying
	$0\le\phi\le1$ and $|\nabla\phi|\le L/R$,
	with the support of its gradient contained in a set $\A$ of finite volume.
	Set
	\[
	E=\int_M u^{-a-1}z^m\phi^b,\qquad
	I=\int_M f\phi^b,\qquad
	J=\int_{\A}f\phi^b.
	\]
	Dropping the source term from \eqref{eq:basicnegative}, we obtain
	\begin{equation}\label{eq:G1first}
		aE\le b\int_{\A}u^{-a}z^{m-1}\phi^{b-1}|\nabla\phi|.
	\end{equation}
	The essential factorization is
	\begin{equation}\label{eq:G1factor1}
		u^{-a}z^{m-1}
		=(u^{-a-1}z^m)^\alpha (u^p z^q)^\beta.
	\end{equation}
	Indeed, the exponent of $u$ satisfies
	\begin{align*}
		-(a+1)\alpha+p\beta
		&=\frac{-(a+1)[p(m-1)+aq]+p(m-1-a)}K\\
		&=\frac{-amp-aq(a+1)}K=-a;
	\end{align*}
	the exponent of $z$ satisfies
	\begin{align*}
		m\alpha+q\beta
		&=\frac{mp(m-1)+maq+q(m-1-a)}K\\
		&=\frac{(m-1)[mp+q(a+1)]}K=m-1.
	\end{align*}
	Including the cutoff powers in this factorization gives
	\[
	u^{-a}z^{m-1}\phi^{b-1}|\nabla\phi|
	=(u^{-a-1}z^m\phi^b)^\alpha
	(f\phi^b)^\beta
	\phi^{b\lambda_0-1}|\nabla\phi|.
	\]
	Applying H\"older's inequality with the first set of three exponents yields
	\begin{align}
		aE
		&\le b E^\alpha J^\beta
		\left(\int_{\A}
		\phi^{b-1/\lambda_0}
		|\nabla\phi|^{1/\lambda_0}\right)^{\lambda_0}\notag\\
		&\le \frac{bL}{R}E^\alpha J^\beta\mu(\A)^{\lambda_0}.
		\label{eq:G1energy}
	\end{align}
	The condition $b\ge1/\lambda_0$ ensures that the remaining cutoff power
	is nonnegative. Since $0\le\phi\le1$, this factor can be discarded
	without introducing a singularity where $\phi=0$.
	If $E>0$, division by $aE^\alpha$ gives
	\begin{equation}\label{eq:G1Esolve}
		E\le
		\left(\frac{bL}{aR}\right)^{1/(1-\alpha)}
		J^{\beta/(1-\alpha)}
		\mu(\A)^{\lambda_0/(1-\alpha)}.
	\end{equation}
	
	\paragraph{Step 3: eliminating the energy by the ordinary cutoff test.}
	By \eqref{eq:basicordinary},
	\[
	I\le b\int_{\A}z^{m-1}\phi^{b-1}|\nabla\phi|.
	\]
	The second factorization is
	\begin{equation}\label{eq:G1factor2}
		z^{m-1}=(u^{-a-1}z^m)^\theta f^\eta.
	\end{equation}
	Here
	\[
	-(a+1)\theta+p\eta=0,\qquad
	m\theta+q\eta
	=\frac{(m-1)[mp+q(a+1)]}{K}=m-1.
	\]
	The second set of H\"older exponents therefore gives
	\begin{align}
		I
		&\le bE^\theta J^\eta
		\left(\int_{\A}
		\phi^{b-1/\kappa_0}|\nabla\phi|^{1/\kappa_0}
		\right)^{\kappa_0}\notag\\
		&\le\frac{bL}{R}E^\theta J^\eta\mu(\A)^{\kappa_0}.
		\label{eq:G1second}
	\end{align}
	If $E=0$, then, since $\theta>0$, this inequality already implies $I=0$.
	It remains to consider $E>0$.
	Substituting \eqref{eq:G1Esolve} into \eqref{eq:G1second}, we obtain
	\begin{align}
		I\le&
		\frac{bL}{R}
		\left(\frac{bL}{aR}\right)^{\theta/(1-\alpha)}
		J^{\,\eta+\theta\beta/(1-\alpha)}
		\mu(\A)^{\,\kappa_0+\theta\lambda_0/(1-\alpha)}.
		\label{eq:G1beforecancel}
	\end{align}
	We check each resulting exponent:
	\begin{equation}\label{eq:G1cancel}
		\begin{aligned}
			\frac{\theta}{1-\alpha}
			&=\frac{p(m-1)}{p+q},\\
			\eta+\frac{\theta\beta}{1-\alpha}
			&=\frac{m-1}{K}
			\left[(a+1)+\frac{p(m-1-a)}{p+q}\right]
			=\frac{m-1}{p+q},\\
			\kappa_0+\frac{\theta\lambda_0}{1-\alpha}
			&=1-\left(\eta+\frac{\theta\beta}{1-\alpha}\right)
			=\frac{\delta}{p+q},\\
			1+\frac{\theta}{1-\alpha}
			&=\frac{mp+q}{p+q}.
		\end{aligned}
	\end{equation}
	The second line uses
	$(a+1)(p+q)+p(m-1-a)=K$.
	The third line can also be verified by taking a common denominator;
	using the fact that each set of weights sums to $1$ makes their
	complementary relation more transparent.
	We conclude that
	\begin{equation}\label{eq:G1local}
		\int_M f\phi^b
		\le C R^{-(mp+q)/(p+q)}
		\mu(\A)^{\delta/(p+q)}
		\left(\int_{\A}f\phi^b\right)^{(m-1)/(p+q)},
	\end{equation}
	where one may take
	\[
	C=bL\left(\frac{bL}{a}\right)^{p(m-1)/(p+q)}.
	\]
	Since $a,b,L$ are fixed, $C$ is independent of $R$.
	More significantly, $a$ has disappeared from all the integral densities.
	
	\paragraph{Step 4: the boundary case $p=0$.}
	In this case $m-1<q<m$.
	We do not formally treat the zero weight $\theta=0$ in
	\eqref{eq:G1weights} as furnishing an ordinary H\"older exponent.
	Instead, set $f=z^q$ directly and use the ordinary test to obtain
	\[
	I\le b\int_{\A}
	(z^q\phi^b)^{(m-1)/q}
	\phi^{b(q-m+1)/q-1}|\nabla\phi|.
	\]
	The conjugate exponents are
	\[
	\frac q{m-1}>1,\qquad
	\frac q{q-m+1}>1,\qquad
	\frac{m-1}{q}+\frac{q-m+1}{q}=1.
	\]
	Choosing $b\ge q/(q-m+1)$, we obtain
	\begin{equation}\label{eq:G1pzero}
		I\le\frac{bL}{R}J^{(m-1)/q}
		\mu(\A)^{(q-m+1)/q}.
	\end{equation}
	This is precisely \eqref{eq:G1local} with $p=0$.
	
	\paragraph{Step 5: summation across scales and conclusion.}
	Take $\phi=\phi_R$, $\A=\A_R$, and $L=1/2$,
	and use $\mu(\A_R)\le V(4R)$.
	Set
	\[
	\gamma=\frac{\delta}{m-1},\qquad
	Q=\frac{mp+q}{m-1}.
	\]
	Since $p+q=(m-1)(1+\gamma)$,
	\eqref{eq:G1local} is exactly the local estimate required in
	Lemma \ref{lem:annular}.
	Moreover, $mp+q=(m-1)p+(p+q)>0$, so $Q>0$.
	Condition \eqref{eq:G1int} is the integral hypothesis of that lemma,
	and consequently $f=0$ almost everywhere.
	If $q>0$, the fact that $u>0$ gives $\nabla u=0$ almost everywhere;
	then $u\in C^1$ and the connectedness of $M$ imply that $u$ is constant.
	If $q=0$, we have $f=u^p>0$, a contradiction.
	If $q<0$, admissibility gives $z>0$ almost everywhere,
	so again $f>0$ almost everywhere, a contradiction.
	This proves Theorem \ref{thm:G1}.
	
	\subsection{The region with a negative power of the function: closure with a fixed weight}
	We now prove Theorem \ref{thm:G3}.
	Let $p<0,m-1<q<m$, and set
	\begin{equation}\label{eq:G3density}
		\delta_3=q-m+1>0,\qquad
		a_0=-\frac{p(m-1)}{\delta_3}>0,\qquad
		F=u^{p-a_0}z^q=u^{pq/\delta_3}z^q.
	\end{equation}
	Since $u$ has a positive lower bound on each compact set,
	we have $F\in L^1_{\loc}(M)$.
	Testing with $u^{-a_0}\phi^b$, \eqref{eq:basicnegative} becomes
	\begin{equation}\label{eq:G3basic}
		\int_M F\phi^b
		+a_0\int_Mu^{-a_0-1}z^m\phi^b
		\le b\int_{\A}u^{-a_0}z^{m-1}\phi^{b-1}|\nabla\phi|.
	\end{equation}
	The choice of $a_0$ is determined by the equation that eliminates
	the extra power of the function:
	\[
	(p-a_0)\frac{m-1}{q}=-a_0
	\quad\Longleftrightarrow\quad
	p(m-1)=-a_0(q-m+1).
	\]
	Therefore
	\begin{equation}\label{eq:G3factor}
		u^{-a_0}z^{m-1}=F^{(m-1)/q}.
	\end{equation}
	Discarding the nonnegative energy term on the left,
	choosing $b\ge q/\delta_3$, and applying the conjugate exponents
	$q/(m-1)$ and $q/\delta_3$, we obtain
	\begin{align}
		\int_M F\phi^b
		&\le b\left(\int_{\A}F\phi^b\right)^{(m-1)/q}
		\left(\int_{\A}
		\phi^{b-q/\delta_3}|\nabla\phi|^{q/\delta_3}
		\right)^{\delta_3/q}\notag\\
		&\le\frac{bL}{R}
		\left(\int_{\A}F\phi^b\right)^{(m-1)/q}
		\mu(\A)^{\delta_3/q}.
		\label{eq:G3local}
	\end{align}
	In Lemma \ref{lem:annular}, take
	\[
	\gamma=\frac{\delta_3}{m-1},\qquad Q=\frac q{m-1}.
	\]
	Since $Q/(1+\gamma)=1$,
	\eqref{eq:G3local} has exactly the exponent structure required by the lemma.
	Condition \eqref{eq:G3int} forces $F=0$, and hence $z=0$ and $u$ is constant.
	
	The sign of $p+q-m+1$ has not been used.
	When $p<0$, the role of $p$ is to determine the fixed test weight $a_0$,
	whereas the volume exponents in the annular integral depend only on $q$ and $m$.
	This also explains why the volume condition in Theorem \ref{thm:G3}
	does not involve $p$.

	\begin{remark}[Relation with the pure gradient inequality]\label{rem:Goffi}
		Let $d=q-m+1>0$ and $p<0$. Choose $H'(s)=s^{p/d}$, so that
		$H(s)=s^{1+p/d}/(1+p/d)$ if $1+p/d\ne0$, and $H(s)=\log s$
		otherwise. The function $H$ is increasing and concave on $(0,\infty)$.
		The weak chain rule, justified by positivity on compact sets, gives
		for $v=H(u)$
		\[
		-\Delta_m v
		\ge H'(u)^{m-1}u^p|\nabla u|^q
		=H'(u)^q|\nabla u|^q=|\nabla v|^q.
		\]
		The omitted term is
		$-(m-1)H'(u)^{m-2}H''(u)|\nabla u|^m\ge0$.
		Although $v$ need not be positive, this is allowed in
		\cite[Theorem 1.2]{Goffi2023}. Its area criterion and the
		volume-to-area implication in \cite[Remark 4.5]{Goffi2023} apply with
		$\gamma=(q-m+1)/(m-1)$ and yield \eqref{eq:G3int}.
		When $p=0$, the same connection applies directly to
		Theorem~\ref{thm:G1}. Thus these specializations are recoveries of
		known pure gradient results; the fixed-density proof makes the
		volume summation common to the mixed and pure gradient problems.
	\end{remark}
	
	\subsection{A complete first-order testing proof of the capacity condition}
	In this subsection we prove Proposition \ref{prop:G2}.
	Suppose that $u>0$ and $-\Delta_m u\ge0$.
	Testing with $u^{1-m}\phi^m$ and writing $w=\log u$, we obtain
	\begin{equation}\label{eq:capacityenergy0}
		(m-1)\int_M|\nabla w|^m\phi^m
		\le m\int_M|\nabla w|^{m-1}\phi^{m-1}|\nabla\phi|.
	\end{equation}
	Denote the energy on the left, without its coefficient, by $E_\phi$.
	The conjugate exponents $m/(m-1)$ and $m$ give
	\[
	(m-1)E_\phi
	\le m E_\phi^{(m-1)/m}
	\left(\int_M|\nabla\phi|^m\right)^{1/m}.
	\]
	If $E_\phi=0$, the desired estimate is immediate.
	Otherwise, divide by its $(m-1)/m$ power and then raise to the power $m$ to get
	\begin{equation}\label{eq:capacityenergy}
		E_\phi\le\left(\frac m{m-1}\right)^m
		\int_M|\nabla\phi|^m.
	\end{equation}
	
	We next give a detailed construction of cutoffs whose energy on the right
	tends to zero. Fix $R_0>0$ and set
	\begin{equation}\label{eq:capacityweights}
		R_j=2^jR_0,\qquad
		b_j=\frac{V(2R_j)}{R_j^m},\qquad
		w_j=b_j^{-1/(m-1)},\qquad
		s=\frac m{m-1}.
	\end{equation}
	Let
	\[
	J_j=\int_{R_j}^{2R_j}
	\left(\frac R{V(R)}\right)^{1/(m-1)}\dd R,
	\qquad c_m=\frac{2^s-1}{s}.
	\]
	The monotonicity of volume gives
	\[
	\frac1{V(2R_j)^{1/(m-1)}}
	\int_{R_j}^{2R_j}R^{1/(m-1)}\dd R
	\le J_j
	\le
	\frac1{V(R_j)^{1/(m-1)}}
	\int_{R_j}^{2R_j}R^{1/(m-1)}\dd R.
	\]
	Evaluating the one-dimensional integral and using
	$V(R_j)=V(2R_{j-1})$, we obtain
	\begin{equation}\label{eq:capacitycompare}
		c_m w_j\le J_j\le c_m 2^s w_{j-1}\qquad(j\ge1).
	\end{equation}
	Thus \eqref{eq:G2int} is equivalent to $\sum_jw_j=\infty$.
	This step requires no volume doubling assumption.
	
	Fix the index $j_0$ of the inner ball.
	For $N\ge j_0$, set
	\[
	S_N=\sum_{j=j_0}^Nw_j,\qquad d_j=\frac{w_j}{S_N}.
	\]
	On $B_{R_{j_0}}$, set $\phi_N=1$;
	on each annulus $R_j<r<2R_j$, let it decrease linearly by $d_j$;
	and for $r\ge2R_N$, set $\phi_N=0$.
	Since $\sum_{j=j_0}^Nd_j=1$, this defines a continuous, compactly supported
	Lipschitz cutoff.
	On each annulus, $|\nabla\phi_N|\le d_j/R_j$, and therefore
	\begin{align}
		\int_M|\nabla\phi_N|^m
		&\le\sum_{j=j_0}^N\frac{V(2R_j)}{R_j^m}d_j^m
		=S_N^{-m}\sum_{j=j_0}^N b_jw_j^m\notag\\
		&=S_N^{-m}\sum_{j=j_0}^Nw_j
		=S_N^{1-m}\longrightarrow0.
		\label{eq:capacityzero}
	\end{align}
	The reciprocal-power weights satisfy exactly $b_jw_j^m=w_j$,
	which explains the choice of the decreases above.
	Substituting \eqref{eq:capacityzero} into \eqref{eq:capacityenergy}
	and using $\phi_N=1$ on the inner ball, we obtain
	\[
	\int_{B_{R_{j_0}}}|\nabla\log u|^m=0.
	\]
	The inner ball can be enlarged arbitrarily.
	Hence $u$ is constant, proving the proposition.
	
	\subsection{Iterated logarithms, general slowly growing factors, and specific parameters}
	We first prove Corollary \ref{cor:logs}.
	Choose $R_0$ sufficiently large that all the functions $L_j(R)$ under
	consideration are positive for $R\ge R_0$.
	By \eqref{eq:logs},
	\begin{equation}\label{eq:loglower}
		\left(\frac{R^{D_i}}{V(R)}\right)^{1/B_i}\frac1R
		\ge \frac{C^{-1/B_i}}
		{R L_1(R)\cdots L_{N-1}(R)L_N(R)^{\beta/B_i}}.
	\end{equation}
	Repeated application of the chain rule yields
	\begin{equation}\label{eq:logderivative}
		L_N'(R)=\frac1{R L_1(R)\cdots L_{N-1}(R)}.
	\end{equation}
	For $N=1$, the empty product on the right is understood to be $1$.
	With $t=L_N(R)$, the comparison integral is exactly
	\[
	C^{-1/B_i}\int_{L_N(R_0)}^\infty t^{-\beta/B_i}\dd t.
	\]
	If $\beta<B_i$, an antiderivative is
	$t^{1-\beta/B_i}/(1-\beta/B_i)$, which tends to infinity.
	If $\beta=B_i$, an antiderivative is $\log t$, which also tends to infinity.
	Thus \eqref{eq:unified} diverges, proving the corollary.
	
	The same method is not restricted to logarithms.
	If $W(R)>0$ is measurable and
	\[
	V(R)\le C R^{D_i}W(R)^{B_i},
	\qquad \int^\infty\frac{\dd R}{R W(R)}=\infty,
	\]
	then the corresponding Liouville property still holds.
	No regular variation assumption is imposed on $W$, and it is not differentiated.
	Iterated logarithms are simply a particularly explicit class of examples
	of this general observation.
	
	If the volume bound has the form
	\[
	V(R)\le C R^{D_i}\prod_{j=1}^N L_j(R)^{b_j},
	\]
	the comparison integral is
	\[
	\int^\infty
	\frac{\dd R}{R\prod_{j=1}^N L_j(R)^{b_j/B_i}}.
	\]
	Following the order $L_1,L_2,\ldots,L_N$, locate the first index
	for which $b_j\ne B_i$.
	After making the changes of variables corresponding to the preceding levels,
	the integral has the form $\int^\infty t^{-b_j/B_i}$
	times finitely many iterated logarithmic factors.
	For every $\varepsilon>0$, any fixed finite product of logarithmic factors,
	and its reciprocal, lie between $t^{-\varepsilon}$ and $t^\varepsilon$
	up to positive multiplicative constants.
	The comparison integral therefore diverges if the first unequal exponent
	is smaller than $B_i$, and converges if it is larger than $B_i$.
	It also diverges when all exponents equal $B_i$.
	Convergence of this comparison integral only means that the given volume
	upper bound does not guarantee divergence; it does not allow one to infer
	convergence of the actual volume integral from that upper bound.
	
	\begin{example}[Semilinear inequalities]\label{ex:semilinear}
		Take $m=2,q=0,p=\sigma>1$. Then
		\[
		D_1=\frac{2\sigma}{\sigma-1},\qquad
		B_1=\frac1{\sigma-1}.
		\]
		The condition in Theorem \ref{thm:G1} is
		\[
		\int^\infty\frac{R^{2\sigma-1}}{V(R)^{\sigma-1}}\dd R=\infty.
		\]
		In particular, when $\sigma=2$,
		\[
		V(R)\le C R^4\log R\log\log R\log\log\log R
		\]
		already suffices to exclude positive solutions,
		whereas \eqref{eq:GSvolume} displays only the bound $R^4\log R$.
		The resulting comparison integral is
		\[
		\int^\infty\frac{\dd R}{R\log R\log\log R\log\log\log R},
		\]
		whose divergence follows directly from three successive logarithmic
		changes of variables.
	\end{example}
	
	\begin{example}[Negative powers of the function and of the gradient]
		For $m=2,p=-1,q=3/2$, the parameters lie in $G_3$.
		We have $D_3=3,B_3=2$, the fixed negative-power parameter is $a_0=2$,
		and the density is $F=u^{-3}|\nabla u|^{3/2}$.
		Thus $V(R)\lesssim R^3(\log R)^2(\log\log R)^2$ is sufficient.
		For $m=2,p=3,q=-1$, the parameters lie in $G_1$,
		and $D_1=5,B_1=1$.
		Within the original admissible solution class for negative gradient powers,
		$V(R)\lesssim R^5\log R\log\log R$ excludes positive solutions.
		The latter example shows that a negative gradient exponent does not obstruct
		the summation of annular masses for a fixed density,
		but the local integrability convention remains essential.
	\end{example}
	
	\begin{example}[Verification for Euclidean volume]
		If $M=\R^n$, then $V(R)=\omega_nR^n$.
		For \eqref{eq:unified}, the comparison integral becomes
		\[
		C\int^\infty R^{(D_i-n)/B_i-1}\dd R,
		\]
		which diverges if and only if $n\le D_i$.
		For $m=2,q=0,p>1$, the condition
		$n\le2p/(p-1)$ is equivalent, when $n>2$, to
		$p\le n/(n-2)$, recovering the familiar threshold for the semilinear inequality.
		This is the threshold for an inequality and should not be confused with
		other critical exponents for the corresponding equation.
	\end{example}

	\section{Subhomogeneous reactions and zero lower growth}\label{sec:subhom}
	This section proves Theorem \ref{thm:G46}.
	The parameters common to the two regions are
	\begin{equation}\label{eq:G46parameters}
		h=m-1-p-q>0,\qquad k=m-q\ge1.
	\end{equation}
	The case $k=1$ corresponds to $G_4$, whereas $k>1$ corresponds to $G_6$.
	The negative-power parameter $a$ will tend to infinity, so factors depending on $a$
	must not be absorbed into unspecified constants.
	
	\paragraph{Step 1: normalization on a fixed compact set.}
	Suppose, to the contrary, that there exists a positive, nonconstant admissible weak solution $u$.
	Since $u\in C^1(M)$ and $M$ is connected,
	there is a point at which $|\nabla u|>0$.
	By continuity, we can choose a compact set $E$ of positive volume
	such that $|\nabla u|\ge c_E'>0$ on $E$.
	Choose $K_0>0$ so that $u\le K_0$ on $E$, and set
	\begin{equation}\label{eq:normalize}
		v=u/K_0,\qquad \ell_0=K_0^{-h}.
	\end{equation}
	The homogeneity of the $m$-Laplace operator gives
	\[
	-\Delta_m v
	=K_0^{-(m-1)}(-\Delta_m u)
	\ge K_0^{p+q-m+1}v^p|\nabla v|^q
	=\ell_0 v^p|\nabla v|^q.
	\]
	On $E$, we have $0<v\le1$, and $|\nabla v|$ has a positive lower bound.
	This fixed set will provide a positive mass independent of both $a$ and $R$.
	
	\paragraph{Step 2: the large parameter and conjugate exponents.}
	Choose $a>\max\{m-1,p+1\}$ sufficiently large, and write
	\begin{equation}\label{eq:largeindices}
		A_a=ka-mp-q,\qquad
		T=\frac{A_a}{h},\qquad
		\rho=\frac{a-p-q}{h},\qquad
		t=\frac{A_a}{a-p-q}=\frac T\rho.
	\end{equation}
	Since
	\[
	a-p-q=(a-m+1)+h>h,
	\]
	we have $\rho>1$. Moreover,
	\begin{equation}\label{eq:tminusone}
		t-1=\frac{(k-1)a-p(m-1)}{a-p-q}.
	\end{equation}
	In $G_6$, we have $k>1$, so this expression is positive for sufficiently large $a$.
	In $G_4$, we have $k=1,p<0$, so it is positive for every $a$ under consideration.
	Thus $t>1$, and we may use the conjugate exponents
	\[
	s=\frac t{t-1},\qquad t,\qquad \frac1s+\frac1t=1.
	\]
	We also define
	\begin{equation}\label{eq:vartheta}
		\vartheta=\frac{a-m+1}{a-p-q}\in(0,1),
		\qquad 1-\vartheta=\frac h{a-p-q}=\frac1\rho.
	\end{equation}
	
	\paragraph{Step 3: the full Young inequality calculation.}
	Take $b\ge T$ and $b\ge1$, test with $v^{-a}\phi^b$, and denote
	\[
	I_a=\int_M v^{p-a}|\nabla v|^q\phi^b,\qquad
	E_a=\int_M v^{-a-1}|\nabla v|^m\phi^b.
	\]
	The basic estimate is
	\begin{equation}\label{eq:large-basic}
		\ell_0 I_a+aE_a
		\le b\int_Mv^{-a}|\nabla v|^{m-1}
		\phi^{b-1}|\nabla\phi|.
	\end{equation}
	To identify every exponent explicitly, put $z=|\nabla v|$ and factor the
	integrand on the right pointwise by setting
	\begin{align*}
		X&=\left(\frac a2 v^{-a-1}z^m\phi^b\right)^{1/s},\\
		Y&=b\left(\frac2a\right)^{1/s}
		v^{-a+(a+1)/s}z^{m-1-m/s}
		\phi^{b-1-b/s}|\nabla\phi|.
	\end{align*}
	Then $XY=bv^{-a}z^{m-1}\phi^{b-1}|\nabla\phi|$.
	Using $XY\le X^s/s+Y^t/t\le X^s+Y^t$, we obtain
	\[
	XY\le
	\frac a2 v^{-a-1}z^m\phi^b+
	2^{t-1}b^ta^{1-t}
	v^{-a+t-1}z^{m-t}\phi^{b-t}|\nabla\phi|^t.
	\]
	The individual exponents follow from
	\begin{align*}
		\frac ts&=t-1,\\
		-at+(a+1)\frac ts&=-a+t-1,\\
		(m-1)t-m\frac ts&=m-t,\\
		(b-1)t-b\frac ts&=b-t.
	\end{align*}
	Substitution into \eqref{eq:large-basic}, followed by absorption of half of the
	energy term, yields
	\begin{equation}\label{eq:large-young}
		\ell_0 I_a+\frac a2 E_a
		\le 2^{t-1}b^ta^{1-t}
		\int_Mv^{-a+t-1}z^{m-t}\phi^{b-t}|\nabla\phi|^t.
	\end{equation}
	When $q<0$, the exponent $m-t$ may be negative.
	The calculation is performed on the full-measure set where $z>0$.
	The H\"older factorization below also ensures that the integral on the right is finite,
	so no global positive lower bound on the gradient is required.
	When $q\ge0$, the corresponding factors at zero-gradient points are interpreted
	by continuity; when $q=0$, the zeroth power is still understood to be $1$.
	
	\paragraph{Step 4: absorption of the remaining integral through the source term.}
	A direct calculation from \eqref{eq:largeindices} gives
	\begin{align}
		-a+t-1
		&=\frac{(p-a)(a-m+1)}{a-p-q}
		=(p-a)\vartheta,\notag\\
		m-t
		&=\frac{q(a-m+1)}{a-p-q}
		=q\vartheta.
		\label{eq:largepowers}
	\end{align}
	Consequently,
	\[
	v^{-a+t-1}z^{m-t}
	=(v^{p-a}z^q)^\vartheta.
	\]
	Including the cutoff powers, we have
	\[
	v^{-a+t-1}z^{m-t}\phi^{b-t}|\nabla\phi|^t
	=(v^{p-a}z^q\phi^b)^\vartheta
	\phi^{b(1-\vartheta)-t}|\nabla\phi|^t.
	\]
	The conjugate H\"older exponents are $1/\vartheta$ and $\rho=1/(1-\vartheta)$.
	Since $t\rho=T$, it follows that
	\begin{align}
		&\int_Mv^{-a+t-1}z^{m-t}\phi^{b-t}|\nabla\phi|^t\notag\\
		&\qquad\le I_a^\vartheta
		\left(\int_M\phi^{b-T}|\nabla\phi|^T\right)^{1/\rho}
		\le I_a^\vartheta
		\left(\int_M|\nabla\phi|^T\right)^{1/\rho}.
		\label{eq:largeholder}
	\end{align}
	The second inequality uses precisely the condition $b\ge T$.
	If $I_a=0$, the desired upper bound is immediate.
	Otherwise, inserting \eqref{eq:largeholder} into \eqref{eq:large-young},
	discarding the energy term, and dividing by $I_a^\vartheta$, we obtain
	\[
	\ell_0 I_a^{1/\rho}
	\le 2^{t-1}b^ta^{1-t}
	\left(\int_M|\nabla\phi|^T\right)^{1/\rho}.
	\]
	Raising both sides to the power $\rho$ gives an estimate with no hidden parameter dependence:
	\begin{equation}\label{eq:largeexact}
		I_a\le
		\ell_0^{-\rho}
		2^{(t-1)\rho}b^Ta^{(1-t)\rho}
		\int_M|\nabla\phi|^T.
	\end{equation}
	
	\paragraph{Step 5: the cutoff and the optimal parameter scale.}
	In this subsection, take $\phi=1$ on $B_R$ and let it decrease linearly to zero
	on $R<r<2R$.
	Then $|\nabla\phi|\le R^{-1}$, and
	\[
	\int_M|\nabla\phi|^T\le R^{-T}V(2R).
	\]
	Fix $B>k/h$ and take $b=Ba$.
	For all sufficiently large $a$, we indeed have $b\ge T$ and $b\ge1$.
	We now simplify every factor in \eqref{eq:largeexact}:
	\begin{align}
		\rho&=\frac{a-p-q}{h},&
		T&=\frac{ka-mp-q}{h},\notag\\
		(t-1)\rho&=\frac{(k-1)a-p(m-1)}h,&
		(1-t)\rho&=\frac{p(m-1)-(k-1)a}h,\notag\\
		T+(1-t)\rho&=\rho,&
		\ell_0^{-\rho}&=K_0^{a-p-q}.
		\label{eq:largeconstantpowers}
	\end{align}
	Since $b^T=B^Ta^T$, the pure powers of $a$ multiply to exactly $a^\rho$.
	We therefore obtain
	\begin{equation}\label{eq:large-beforechoice}
		\int_{B_R}v^{p-a}z^q
		\le C_0 V(2R)a^{-(p+q)/h}R^{(mp+q)/h}
		\left(\frac{C_1K_0 a^{1/h}}{R^{k/h}}\right)^a,
	\end{equation}
	with the explicit constants
	\begin{equation}\label{eq:largeconstants}
		C_0=K_0^{-p-q}2^{-p(m-1)/h}B^{-(mp+q)/h},
		\qquad C_1=2^{(k-1)/h}B^{k/h}.
	\end{equation}
	They depend only on the fixed quantities $m,p,q,B,K_0$, and not on $a,R$.
	
	Choose a fixed, sufficiently small $\varepsilon>0$ such that
	\[
	C_1K_0\varepsilon^{1/h}\le e^{-2},
	\]
	and set
	\begin{equation}\label{eq:optimal-a}
		a=\varepsilon R^k.
	\end{equation}
	The parameter $a$ may be real; the test does not require it to be an integer.
	For sufficiently large $R$, all the preceding restrictions on $a$ hold.
	After substitution, the power of the radius outside the parentheses is computed exactly as
	\begin{align}
		\frac{mp+q-k(p+q)}h
		&=\frac{mp+q-(m-q)(p+q)}h\notag\\
		&=\frac{q(p+q-m+1)}h=-q.
		\label{eq:radiuscancel}
	\end{align}
	Thus \eqref{eq:large-beforechoice} reduces to
	\begin{equation}\label{eq:large-final}
		\int_{B_R}v^{p-a}z^q
		\le C_0\varepsilon^{-(p+q)/h}
		V(2R)R^{-q}\exp(-2\varepsilon R^k).
	\end{equation}
	This calculation identifies $R^{m-q}$ as the appropriate large-parameter scale,
	rather than merely a scale proportional to $R$.
	
	\paragraph{Step 6: a fixed positive mass and the lower-limit condition.}
	For sufficiently large $R$, we have $E\subset B_R$.
	Since $v\le1$ on $E$ and $a\ge p+1$,
	\[
	v^{p-a}\ge v^{-1}\quad\text{on }E.
	\]
	Hence
	\begin{equation}\label{eq:fixedmass}
		0<c_E:=\int_Ev^{-1}z^q
		\le\int_{B_R}v^{p-a}z^q.
	\end{equation}
	The constant $c_E$ is finite and strictly positive:
	$E$ is compact, and both $v$ and $z$ have positive lower bounds on $E$.
	Combining this with \eqref{eq:large-final} gives
	\begin{equation}\label{eq:contradictionbound}
		c_E\le C_0\varepsilon^{-(p+q)/h}
		V(2R)R^{-q}\exp(-2\varepsilon R^k).
	\end{equation}
	The condition $h_k(M)=0$ provides a sequence $S_j\to\infty$
	such that $\log V(S_j)/S_j^k\to0$.
	Set $R_j=S_j/2$. Then
	$\log V(2R_j)=o(R_j^k)$.
	The logarithm of the right-hand side of \eqref{eq:contradictionbound},
	apart from a fixed additive constant, is
	\[
	\log V(2R_j)-q\log R_j-2\varepsilon R_j^k
	\longrightarrow-\infty.
	\]
	The right-hand side therefore tends to zero, contradicting $c_E>0$.
	This excludes nonconstant positive solutions.
	
	If $q>0$, the remaining positive constant solutions are exactly those permitted by
	the Liouville conclusion.
	If $q=0$, a positive constant $u=c>0$ gives zero on the left-hand side
	and $c^p>0$ on the right, and thus does not satisfy the weak inequality.
	If $q<0$, positive constants do not belong to the admissible solution class.
	Consequently, no positive admissible solution exists when $q\le0$ either.
	This proves Theorem \ref{thm:G46}.
	
	\begin{remark}\label{rem:criticalG6}
		The permissible size of $\varepsilon$ in \eqref{eq:optimal-a}
		depends on the normalization constant $K_0$ obtained on the fixed compact set.
		Thus $\varepsilon$ cannot be increased arbitrarily to absorb an arbitrary
		prescribed positive coefficient $C$ in $\log V(R)\sim CR^k$.
		This is why the proof yields the condition of zero lower growth rate
		without directly covering all critical growth bounds of the form $\exp(CR^k)$.
		This is an actual restriction of the estimate, rather than an omitted
		constant-absorption step.
	\end{remark}
	
	\section{The homogeneous line and the exact entropy threshold}\label{sec:entropy}
	This section proves the Liouville property in Theorem \ref{thm:G5}.
	Assume $p+q=m-1$ and $p\ge0$, and put
	$w=\log u$ and $z=|\nabla w|$.
	For a nonnegative compactly supported Lipschitz function $\psi$,
	use $u^{1-m}\psi$ as a test function in \eqref{eq:weak}.
	Expanding its derivative gives
	\begin{align*}
		&\int_M|\nabla u|^{m-2}
		\ip{\nabla u}{u^{1-m}\nabla\psi}
		+(1-m)\int_Mu^{-m}|\nabla u|^m\psi\\
		&\hspace{35mm}\ge
		\int_Mu^{p+1-m}|\nabla u|^q\psi .
	\end{align*}
	Since $\nabla u=u\nabla w$, we have
	\[
	u^{1-m}|\nabla u|^{m-2}\nabla u
	=|\nabla w|^{m-2}\nabla w,
	\quad
	u^{-m}|\nabla u|^m=z^m,
	\]
	and
	$u^{p+1-m}|\nabla u|^q=u^{p+q-m+1}z^q=z^q$.
	Therefore
	\begin{equation}\label{eq:logweak}
		\int_M|\nabla w|^{m-2}\ip{\nabla w}{\nabla\psi}
		\ge \int_M\bigl[(m-1)z^m+z^q\bigr]\psi .
	\end{equation}
	This is the weak form of
	$-\Delta_m w\ge(m-1)|\nabla w|^m+|\nabla w|^q$.
	
	Write $F(t)=(m-1)t^m+t^q$.
	For $t>0$, the identity $q=m-1-p$ gives
	\begin{equation}\label{eq:Fquotient}
		\frac{F(t)}{t^{m-1}}=(m-1)t+t^{-p}.
	\end{equation}
	If $p>0$, the right-hand side tends to infinity at both endpoints,
	and its derivative is $(m-1)-pt^{-p-1}$.
	Its unique minimizer is
	\[
	t_*=\left(\frac p{m-1}\right)^{1/(p+1)}.
	\]
	Using $(m-1)t_*=pt_*^{-p}$, we find
	\[
	\inf_{t>0}\bigl[(m-1)t+t^{-p}\bigr]
	=(p+1)t_*^{-p}
	=(p+1)\left(\frac{m-1}{p}\right)^{p/(p+1)}.
	\]
	If $p=0$, the right-hand side is $(m-1)t+1$, whose infimum is $1$.
	Consequently, for all admissible values of $t$,
	\begin{equation}\label{eq:Fbounds}
		F(t)\ge c_*(m,p)t^{m-1},\qquad F(t)\ge(m-1)t^m.
	\end{equation}
	When $q<0$, these inequalities are used on the full-measure set where $z>0$;
	when $q=0$, we have $F(0)=1$; and when $q>0$, we have $F(0)=0$.
	
	Suppose that $h_1(M)<c_*(m,p)$.
	Choose
	\[
	h_1(M)<\sigma<c_*(m,p),\qquad
	d=1-\frac{\sigma}{c_*(m,p)}>0.
	\]
	Let $\eta_R=1$ on $B_R$, let it decrease linearly to zero on $R<r<R+1$,
	and set it equal to zero outside this region.
	Then $|\nabla\eta_R|\le1$.
	Test \eqref{eq:logweak} with
	\begin{equation}\label{eq:exp-test}
		\psi=e^{-\sigma r}\eta_R^m
	\end{equation}
	This function is compactly supported, so no global weighted integrability
	is being assumed.
	Using $|\nabla r|\le1$, we obtain
	\begin{align}
		\int_M F(z)e^{-\sigma r}\eta_R^m
		\le{}&\sigma\int_Mz^{m-1}e^{-\sigma r}\eta_R^m\notag\\
		&+m\int_M z^{m-1}e^{-\sigma r}
		\eta_R^{m-1}|\nabla\eta_R|.
		\label{eq:exp-first}
	\end{align}
	The first inequality in \eqref{eq:Fbounds} allows the first term to be absorbed, giving
	\begin{equation}\label{eq:exp-absorb}
		d\int_MF(z)e^{-\sigma r}\eta_R^m
		\le m\int_M(z\eta_R)^{m-1}|\nabla\eta_R|e^{-\sigma r}.
	\end{equation}
	
	To display the constant explicitly, we use the following form of Young's inequality:
	\begin{equation}\label{eq:exactyoungentropy}
		m x^{m-1}y
		\le\frac d2(m-1)x^m+\left(\frac2d\right)^{m-1}y^m,
		\qquad x,y\ge0.
	\end{equation}
	It follows by applying Young's inequality with conjugate exponents
	$m/(m-1)$ and $m$ to suitably scaled factors.
	Alternatively, maximize the left-hand side minus the first term on the right
	as a function of $x$.
	The maximizer satisfies $x=2y/d$, and substitution gives exactly the
	second term on the right.
	Taking $x=z\eta_R,y=|\nabla\eta_R|$, and using
	$F(z)\ge(m-1)z^m$, we deduce from \eqref{eq:exp-absorb} that
	\[
	\frac d2\int_MF(z)e^{-\sigma r}\eta_R^m
	\le\left(\frac2d\right)^{m-1}
	\int_M e^{-\sigma r}|\nabla\eta_R|^m.
	\]
	Hence
	\begin{equation}\label{eq:entropyestimate}
		\int_MF(z)e^{-\sigma r}\eta_R^m
		\le\left(\frac2d\right)^m
		e^{-\sigma R}V(R+1).
	\end{equation}
	The constant may diverge as $\sigma$ approaches the threshold,
	but it is finite for the fixed choice $\sigma<c_*$.
	
	Since $h_1(M)<\sigma$, there exist $\varepsilon>0$ and $S_j\to\infty$
	such that $V(S_j)\le\exp((\sigma-\varepsilon)S_j)$.
	Take $R_j=S_j-1$ in \eqref{eq:entropyestimate}.
	The right-hand side is at most
	\[
	\left(\frac2d\right)^m e^\sigma e^{-\varepsilon S_j}
	\longrightarrow0.
	\]
	For any fixed ball $B_L$, we have $\eta_{R_j}=1$ there for all sufficiently large $j$.
	Thus $\int_{B_L}F(z)e^{-\sigma r}=0$, and consequently $F(z)=0$ almost everywhere.
	It follows from \eqref{eq:Fbounds} that $z=0$.
	If $q>0$, then $u$ is constant.
	If $q=0$, the inequality $F(z)\ge1$ gives a contradiction.
	If $q<0$, admissibility likewise gives a contradiction.
	This proves the nonexistence assertion.
	
	\subsection{Verification for a constant source coefficient}
	We now prove Proposition \ref{prop:coefficient} and explain the role of
	the coefficient in each parameter region.
	Suppose that $-\Delta_m u\ge\lambda u^p|\nabla u|^q$.
	In the first negative-power test in $G_1$, discarding the source term
	still gives the same energy estimate as \eqref{eq:G1energy}.
	The ordinary test has $\lambda I$ on the left.
	Thus the constant in \eqref{eq:G1local} acquires only a factor $\lambda^{-1}$.
	The annular lemma permits any fixed positive constant in the local estimate,
	so neither the volume-integral exponents nor the divergence condition changes.
	The separate proof for $p=0$ likewise introduces only $\lambda^{-1}$.
	
	In $G_3$, the source term in \eqref{eq:G3basic} becomes
	$\lambda\int F\phi^b$, while the energy term is unchanged.
	Discarding the energy term changes only the fixed constant in \eqref{eq:G3local}.
	In $G_2$, we simply discard the source term,
	and Proposition \ref{prop:G2} does not involve $\lambda$ at all.
	
	In $G_4,G_6$, the normalized coefficient is
	$\ell_0=\lambda K_0^{-h}$, so the right-hand side of
	\eqref{eq:largeexact} acquires a factor $\lambda^{-\rho}$.
	Since $\rho=(a-p-q)/h$,
	\[
	\lambda^{-\rho}
	=\lambda^{(p+q)/h}(\lambda^{-1/h})^a.
	\]
	The constant outside the parentheses in \eqref{eq:large-beforechoice}
	therefore acquires the factor $\lambda^{(p+q)/h}$,
	and the fixed constant inside is multiplied by $\lambda^{-1/h}$.
	Both factors are independent of the radius and the negative-power parameter.
	We can still choose a sufficiently small fixed $\varepsilon>0$,
	so the condition of zero lower growth rate remains unchanged.
	
	In $G_5$, the logarithmic transformation yields
	\[
	-\Delta_m w\ge(m-1)|\nabla w|^m
	+\lambda|\nabla w|^q.
	\]
	The required optimal coefficient becomes
	\[
	\inf_{t>0}\bigl[(m-1)t+\lambda t^{-p}\bigr].
	\]
	For $p>0$, its minimizer is
	$t_\lambda=(\lambda p/(m-1))^{1/(p+1)}$,
	and $(m-1)t_\lambda=\lambda p t_\lambda^{-p}$.
	The infimum is therefore
	\[
	\lambda(p+1)t_\lambda^{-p}
	=(p+1)\left(\frac{m-1}{p}\right)^{p/(p+1)}
	\lambda^{1/(p+1)}.
	\]
	For $p=0$, the infimum is $\lambda$.
	Replacing $c_*$ in the preceding subsection by this constant gives
	\eqref{eq:cstar-lambda} and proves the proposition.
	
	\section{Sharpness and global model constructions}\label{sec:models}
	\subsection{Radial flux and global gluing}
	To prove Theorem~\ref{thm:logsharp}, the endpoint assertion in
	Theorem~\ref{thm:G5}, and Proposition~\ref{prop:models}, we first establish the common steps
	of the radial constructions.
	The metrics constructed below are smooth, whereas the functions need only
	have the $C^1$ regularity required in Definition \ref{def:solution}.
	
	Take $n\ge2$ and consider, in polar coordinates on $\R^n$,
	\begin{equation}\label{eq:modelmetric}
		g=\dd r^2+\psi(r)^2 g_{\mathbb S^{n-1}},
		\qquad \psi(r)>0\quad(r>0).
	\end{equation}
	We require $\psi(r)=r$ near the origin, so that the metric extends smoothly across the pole.
	The spherical area density is
	\[
	A(r)=\omega_{n-1}\psi(r)^{n-1},\qquad
	V(R)=\int_0^R A(r)\dd r,
	\]
	where $\omega_{n-1}$ is the area of the unit sphere.
	The radial length element is $\dd r$, so every curve escaping to infinity has infinite length;
	this ensures completeness of the model.
	The distance from the pole to $(r,\theta)$ is exactly $r$:
	the radial segment gives an upper bound of $r$,
	while any connecting curve has length at least the total change $r$
	in its radial coordinate.
	
	For a radial function $u=u(r)$, the radial divergence formula gives
	\begin{equation}\label{eq:radialoperator}
		\Delta_m u=\frac1{A(r)}
		\left(A(r)|u'|^{m-2}u'\right)'.
	\end{equation}
	Set $f(r)=-u'(r)>0$. Then
	\begin{equation}\label{eq:radialflux}
		-\Delta_m u
		=(f^{m-1})'+\frac{A'}A f^{m-1}
		=f^{m-1}\left((m-1)\frac{f'}f+\frac{A'}A\right).
	\end{equation}
	Dividing the desired inequality by the positive quantity $f^{m-1}$,
	we may equivalently write it as
	\begin{equation}\label{eq:areainequality}
		\frac{A'}A\ge
		u^pf^{-(m-1-q)}-(m-1)\frac{f'}f.
	\end{equation}
	This reduces the choice of metric to a one-dimensional inequality.
	
	\begin{lemma}[Global gluing with prescribed growth]\label{lem:glue}
		Suppose that $f\in C^\infty((0,\infty))$, $f>0$,
		that $f(r)=cr^a$ near the origin for some $c,a>0$,
		and that $\int_1^\infty f(r)\dd r<\infty$.
		Given $L\ge0$, define
		\[
		u(r)=L+\int_r^\infty f(t)\dd t.
		\]
		Assume that there exist $0<r_0<r_1$ and a positive smooth area function
		$A_\infty$ defined near infinity such that \eqref{eq:areainequality}
		is satisfied by $A_0(r)=\omega_{n-1}r^{n-1}$ on $0<r<r_0$
		and by $A_\infty(r)$ on $r>r_1$.
		If
		\begin{equation}\label{eq:localintegrabilitymodel}
			n+aq>0,
		\end{equation}
		then there exists a smooth complete model metric with $A=A_0$ near the origin
		and $A=C A_\infty$ for sufficiently large $r$,
		such that the function $u$ above extends to a globally positive
		$C^1$ admissible weak solution.
	\end{lemma}
	
	\begin{proof}
		Denote the right-hand side of \eqref{eq:areainequality} by
		\[
		H(r)=u(r)^pf(r)^{-(m-1-q)}-(m-1)\frac{f'(r)}{f(r)}.
		\]
		Since $u,f>0$ on $(0,\infty)$, the function $H$ is smooth there.
		Put $b_0=(n-1)/r$ and $b_\infty=A_\infty'/A_\infty$.
		Extend the latter smoothly to $(0,\infty)$;
		its values outside the region near infinity do not affect the construction.
		Choose smooth functions $\chi_0,\chi_\infty\in[0,1]$ such that
		$\chi_0=1$ for sufficiently small $r$, with support contained in $(0,r_0)$,
		and $\chi_\infty=1$ for sufficiently large $r$, with support contained in $(r_1,\infty)$.
		Their supports may be chosen disjoint, so $\chi_0+\chi_\infty\le1$.
		Define
		\begin{equation}\label{eq:glueformula}
			b=\chi_0 b_0+\chi_\infty b_\infty
			+(1-\chi_0-\chi_\infty)(H+1).
		\end{equation}
		On the support of $\chi_0$, we have $b_0\ge H$;
		on the support of $\chi_\infty$, we have $b_\infty\ge H$;
		and the function used in the third term is also greater than $H$.
		Thus $b\ge H$ everywhere.
		Choose $\bar r>0$ in the region where $\chi_0=1$, and define
		\[
		A(r)=\omega_{n-1}\bar r^{\,n-1}
		\exp\!\left(\int_{\bar r}^r b(s)\dd s\right).
		\]
		Then $A'/A=b$, and $A=\omega_{n-1}r^{n-1}$ near the origin.
		Near infinity, $b=b_\infty$, so $A/A_\infty$ is a positive constant.
		Taking $\psi=(A/\omega_{n-1})^{1/(n-1)}$ gives the required
		smooth complete metric \eqref{eq:modelmetric}.
		Since $b\ge H$, the desired inequality holds for $r>0$.
		
		It remains to check that the pole introduces no inadmissible singular source.
		Near the origin,
		\[
		u(r)=u(0)-\frac{c}{a+1}r^{a+1},
		\qquad |\nabla u|=cr^a\longrightarrow0.
		\]
		Because $a>0$, the radial function is differentiable at the pole with zero gradient.
		Its gradient is continuous, so $u\in C^1(M)$.
		Moreover, $u(0)>0$, and hence $u^p$ is bounded above and below by
		positive constants near the pole.
		The local integral of the source term is comparable to
		\[
		\int_0^{r_0}r^{n-1+aq}\dd r
		\]
		which is finite by \eqref{eq:localintegrabilitymodel}.
		When $q<0$, the pole is a zero-gradient point, but it is a set of measure zero;
		the calculation above verifies admissibility in the improper-integral sense.
		
		Finally, integrate by parts against a nonnegative smooth test function on the
		region with $B_\varepsilon$ removed.
		The absolute value of the flux through the small sphere about the pole is bounded by
		\[
		\|\psi_0\|_\infty A(\varepsilon)f(\varepsilon)^{m-1}
		=O\!\left(\varepsilon^{n-1+a(m-1)}\right)
		\longrightarrow0
		\]
		where $\psi_0$ denotes the test function.
		The source term passes to the limit by local integrability,
		and the flux paired with the test gradient on the left is locally integrable as well.
		Letting $\varepsilon\downarrow0$ therefore yields the weak inequality on the
		entire model, with no additional Dirac term.
		The approximation argument given earlier extends this inequality to Lipschitz tests,
		completing the proof.
	\end{proof}

	\subsection{An audit of the radial flux and pole conditions}\label{subsec:radialaudit}
	We record the two checks on which every model construction depends.
	For a decreasing radial profile with $f=-u'>0$ and spherical area
	$A$, the distributional flux is $-A f^{m-1}$.  Consequently, on
	$r>0$ the supersolution inequality is equivalent to
	\begin{equation}\label{eq:auditedflux}
		(A f^{m-1})'\ge A u^p f^q,
		\qquad\text{or equivalently}\qquad
		\frac{A'}A+(m-1)\frac{f'}f\ge u^p f^{q-m+1}.
	\end{equation}
	In particular, a strict asymptotic estimate for the ratio of the
	left side of the first inequality to $A u^p f^q$ yields the desired
	inequality for all sufficiently large radii; no asymptotic equality
	alone is used to deduce a global inequality.
	
	At the pole, let $A(r)=\omega_{n-1}r^{n-1}$, $f(r)=cr^a$ and
	$u(0)>0$.  Direct differentiation yields
	\begin{equation}\label{eq:pole-audit}
		\frac{-\Delta_m u}{f^q}
		=c^{m-1-q}\bigl(n-1+a(m-1)\bigr)
		r^{a(m-1-q)-1}.
	\end{equation}
	If $a(m-1-q)<1$, this tends to $+\infty$ as $r\downarrow0$,
	while $u(r)^p\to u(0)^p>0$.  Thus the supersolution inequality
	holds on a punctured neighborhood of the pole, regardless of the
	sign of $p$ or $q$.  At the same time,
	\begin{equation}\label{eq:pole-source-audit}
		\int_{B_\varepsilon}u^p|\nabla u|^q\,d\mu
		\asymp\int_0^\varepsilon r^{n-1+aq}\,dr <\infty
		\quad\Longleftrightarrow\quad n+aq>0.
	\end{equation}
	The flux of the punctured ball is
	$A(\varepsilon)f(\varepsilon)^{m-1}
	=O(\varepsilon^{n-1+a(m-1)})\to0$.
	These are independent conditions: the first checks the differential
	inequality, the second the admissible source, and the third excludes
	a Dirac mass.  They apply unchanged when $q<0$, since the pole has
	measure zero.  In each model below the pole exponent is chosen to
	satisfy both $a(m-1-q)<1$ and $n+aq>0$.
	
	\subsection{Models beyond every finite logarithmic threshold}
	We prove Theorem~\ref{thm:logsharp}.
	All formulas in this subsection are prescribed only for sufficiently
	large $r$, where every logarithm is positive. Put
	\[
	P_N(r)=\prod_{j=1}^{N-1}L_j(r),\qquad T(r)=L_N(r),
	\qquad T'(r)=\frac1{rP_N(r)}.
	\]
	Empty products equal $1$. Prescribe
	\begin{equation}\label{eq:logsharparea}
		A_\infty(r)=r^{D_i-1}P_N(r)^{B_i}T(r)^b.
	\end{equation}
	For the decreasing functions below, write $f=-u'>0$ and
	$J=A_\infty f^{m-1}$. By \eqref{eq:radialoperator}, the required
	inequality is exactly
	\begin{equation}\label{eq:fluxreaction}
		J'(r)\ge A_\infty(r)u(r)^p f(r)^q.
	\end{equation}
	We verify this with a strict inequality near infinity in each region.
	
	\paragraph{The region $G_1$.}
	Set $h=p+q-m+1>0$ and $\beta=(m-q)/h>0$.
	Choose
	\[
	\frac1h<a<\frac b{m-1},\qquad
	u(r)=\varepsilon r^{-\beta}P_N(r)^{-1/h}T(r)^{-a},
	\quad\varepsilon>0.
	\]
	This choice is possible because $b>B_1=(m-1)/h$.
	Since $D_1=m+(m-1)\beta$, differentiation gives
	\[
	f=\frac ur\,B(r),\qquad
	B(r)=\beta+\frac1h\sum_{j=1}^{N-1}
	\frac1{\prod_{l=1}^jL_l(r)}
	+\frac a{P_N(r)T(r)}\longrightarrow\beta.
	\]
	Let $\kappa=b-a(m-1)>0$. Exact cancellation of the power of $r$
	and the first $N-1$ logarithmic factors gives
	\[
	J=\varepsilon^{m-1}B(r)^{m-1}T(r)^\kappa.
	\]
	Here $B'=O(1/(rL_1(r)^2))$.
	For every fixed $N$,
	$P_N(r)T(r)/L_1(r)^2\to0$.
	It follows that differentiation of $B$ is of lower order and
	\[
	J'\sim\varepsilon^{m-1}\beta^{m-1}\kappa
	\frac{T^{\kappa-1}}{rP_N}.
	\]
	On the other hand, the identities
	$D_1-1-p\beta-q(\beta+1)=-1$ and
	$B_1-(p+q)/h=-1$ yield
	\[
	A_\infty u^pf^q
	\sim\varepsilon^{p+q}\beta^q
	\frac{T^{b-a(p+q)}}{rP_N}.
	\]
	Consequently,
	\begin{equation}\label{eq:G1logratio}
		\frac{J'}{A_\infty u^pf^q}
		\sim\varepsilon^{-h}\beta^{m-1-q}\kappa T^{ah-1}
		\longrightarrow\infty.
	\end{equation}
	This proves \eqref{eq:fluxreaction} near infinity, including when
	$q$ is negative.
	
	\paragraph{The region $G_3$.}
	Write $d=q-m+1\in(0,1)$ and choose
	\[
	\frac1d<a<\frac b{m-1},\qquad
	f(r)=C r^{-1/d}P_N(r)^{-1/d}T(r)^{-a},\qquad
	u(r)=L+\int_r^\infty f(t)\dd t,
	\]
	where $L,C>0$. The tail integral converges because $1/d>1$;
	in particular, $u^p\to L^p$.
	As $D_3-1=(m-1)/d$, setting $\kappa=b-a(m-1)>0$ gives exactly
	\[
	J=C^{m-1}T^\kappa,\qquad
	J'=C^{m-1}\kappa\frac{T^{\kappa-1}}{rP_N}.
	\]
	It follows that
	\begin{equation}\label{eq:G3logratio}
		\frac{J'}{A_\infty u^pf^q}
		\sim C^{-d}L^{-p}\kappa T^{ad-1}\longrightarrow\infty.
	\end{equation}
	
	\paragraph{The region $G_2$.}
	Here $q\ge m$, while $p$ is arbitrary.
	Choose $1<a<b/(m-1)$ and set
	\[
	f(r)=\frac C{rP_N(r)T(r)^a},\qquad
	u(r)=L+\int_r^\infty f(t)\dd t
	=L+\frac C{a-1}T(r)^{1-a}.
	\]
	The equality follows from $T'=1/(rP_N)$.
	Again $u^p\to L^p$, and for $\kappa=b-a(m-1)>0$,
	$J=C^{m-1}T^\kappa$. Therefore
	\begin{equation}\label{eq:G2logratio}
		\frac{J'}{A_\infty u^pf^q}
		\sim C^{m-1-q}L^{-p}\kappa\,
		r^{q-m}P_N(r)^{q-m}T(r)^{a(q-m+1)-1}
		\longrightarrow\infty.
	\end{equation}
	For $q=m$ this follows from $a>1$; for $q>m$ the positive power
	of $r$ already suffices.
	
	\paragraph{Globalization and admissibility.}
	In all three cases, choose a smooth positive interpolation of $f$ to
	$f(r)=cr^{a_0}$ near the pole, where
	\[
	a_0>0,\qquad a_0(m-1-q)<1,\qquad n+a_0q>0.
	\]
	If $m-1-q\le0$, the middle restriction is automatic.
	Otherwise choose $a_0$ small; if $q<0$, also require
	$a_0<n/(-q)$. Such a choice always exists.
	For $G_1$ use $u=\int_r^\infty f$; this agrees exactly with the
	prescribed tail. For $G_2,G_3$ retain the positive constant $L$.
	In each case $u(0)$ is finite and positive.
	The right-hand side $H$ of \eqref{eq:areainequality} satisfies
	\[
	rH(r)=O\bigl(r^{1-a_0(m-1-q)}\bigr)-(m-1)a_0
	\longrightarrow-(m-1)a_0<n-1.
	\]
	Thus the Euclidean area density satisfies the inequality near the pole.
	Equations~\eqref{eq:G1logratio}--\eqref{eq:G2logratio} verify it
	near infinity. Lemma~\ref{lem:glue} now gives a smooth complete metric,
	a globally admissible $C^1$ function, and no singular measure at the pole.
	
	Finally, $D_i>0$ and logarithmic differentiation of
	$W(r)=P_N(r)^{B_i}T(r)^b$ gives $rW'/W\to0$.
	Since $A=C_0r^{D_i-1}W(r)$ near infinity, differentiation of
	$R^{D_i}W(R)/D_i$ proves
	\[
	V(R)\sim\frac{C_0}{D_i}R^{D_i}W(R).
	\]
	This is \eqref{eq:logsharpvolume} and completes the proof.
	
	\subsection{Endpoint models on the homogeneous line}
	\paragraph{The entropy endpoint for $p>0$.}
	Given $H>0$, consider the $n$-dimensional hyperbolic space of constant
	curvature $-a^2$, where $a=H/(n-1)$.
	Its spherical area is
	\[
	A_H(r)=\omega_{n-1}\left(\frac{\sinh(ar)}a\right)^{n-1},
	\]
	so
	$\lim_{R\to\infty}R^{-1}\log V(R)=(n-1)a=H$.
	There is a smooth Busemann function $b$ satisfying
	\begin{equation}\label{eq:busemann}
		|\nabla b|=1,\qquad \Delta b=H.
	\end{equation}
	These identities can also be checked directly in global horospherical coordinates:
	the metric is $\dd b^2+e^{2ab}|{\dd y}|^2$,
	the volume density is $e^{(n-1)ab}$,
	and hence $\Delta b=\partial_b\log(e^{(n-1)ab})=H$.
	
	Set $u=e^{-tb}$, where $t>0$.
	Then
	\[
	\nabla u=-tu\nabla b,\qquad
	|\nabla u|=tu,\qquad
	|\nabla u|^{m-2}\nabla u=-t^{m-1}u^{m-1}\nabla b.
	\]
	Taking the divergence of the last expression gives
	\begin{align}
		-\Delta_m u
		&=t^{m-1}\left[
		\ip{\nabla(u^{m-1})}{\nabla b}+u^{m-1}\Delta b\right]\notag\\
		&=t^{m-1}[H-(m-1)t]u^{m-1}.
		\label{eq:hyperboliccompute}
	\end{align}
	In $G_5$, we have $p+q=m-1$, and therefore
	\[
	u^p|\nabla u|^q=t^q u^{m-1}.
	\]
	The inequality is thus equivalent to
	\[
	H\ge(m-1)t+t^{q-m+1}=(m-1)t+t^{-p}.
	\]
	When $p>0$, choose the minimizer $t=t_*$ found above and take $H=c_*(m,p)$.
	This gives a smooth global solution of the corresponding equality.
	The solution is strictly positive and nonconstant, with a gradient that never vanishes,
	so the construction includes $q<0$.
	It proves that the strict threshold in Theorem \ref{thm:G5} cannot be relaxed when $p>0$.
	
	\paragraph{The entropy endpoint for $p=0$.}
	Here $q=m-1$ and $c_*=1$.
	Since the infimum in \eqref{eq:Fquotient} is approached only as $t\downarrow0$,
	the preceding exponential function does not directly give a nonconstant
	equality solution at $H=1$.
	Instead, prescribe the following radial model near infinity:
	\begin{equation}\label{eq:pzero-tail}
		A_\infty(r)=e^r r^B,\qquad
		u(r)=r^{-\alpha},\qquad
		\alpha>0,\quad B\ge(m-1)(\alpha+1).
	\end{equation}
	Then $f=-u'=\alpha r^{-\alpha-1}$,
	$f'/f=-(\alpha+1)/r$, and $A_\infty'/A_\infty=1+B/r$.
	By \eqref{eq:radialflux},
	\begin{equation}\label{eq:pzero-ratio}
		\frac{-\Delta_m u}{|u'|^{m-1}}
		=1+\frac{B-(m-1)(\alpha+1)}r\ge1.
	\end{equation}
	Thus $-\Delta_m u\ge|\nabla u|^{m-1}$ holds near infinity.
	
	For a global construction, choose a positive smooth function $f$,
	with $f(r)=r$ near the origin and $f(r)=\alpha r^{-\alpha-1}$ for sufficiently large $r$,
	with a positive smooth interpolation in between.
	Define $u(r)=\int_r^\infty f$.
	The resulting function is still exactly $r^{-\alpha}$ near infinity,
	while $u=u(0)-r^2/2$ near the origin.
	Since $p=0$, the required area inequality is
	\[
	\frac{A'}A\ge1-(m-1)\frac{f'}f.
	\]
	Near the pole, its right-hand side is $1-(m-1)/r$,
	whereas the Euclidean area density gives $(n-1)/r$.
	Thus the condition holds for sufficiently small $r$.
	It holds near infinity by \eqref{eq:pzero-ratio}.
	Lemma \ref{lem:glue} now gives a globally positive admissible weak solution.
	
	The area near infinity changes only by a positive constant factor, so
	$V(R)\sim C e^R R^B$.
	For example, this follows by comparing the derivatives of the defining volume
	integral and $e^R R^B$, since
	\[
	\frac{(e^R R^B)'}{e^R R^B}=1+\frac BR\longrightarrow1.
	\]
	The volume entropy is therefore exactly $1$.
	This shows that $h_1<1$ cannot be replaced by $h_1\le1$ when $p=0$ either.
	Because the example contains an additional polynomial factor,
	it does not determine whether the more precise bound $V(R)\le Ce^R$ is sufficient.
	All assertions of Theorem \ref{thm:G5} are now proved.
	
	For the constant source coefficient in Proposition \ref{prop:coefficient},
	when $p>0$ take $t=t_\lambda$ and $H=c_{*,\lambda}(m,p)$ in hyperbolic space
	to obtain the corresponding equality solution.
	When $p=0$, replace the area density in \eqref{eq:pzero-tail}
	by $A_\infty(r)=e^{\lambda r}r^B$.
	The right-hand side of \eqref{eq:pzero-ratio} then becomes
	$\lambda+[B-(m-1)(\alpha+1)]/r\ge\lambda$,
	and the volume entropy equals $\lambda$.
	Thus the strict entropy threshold with a constant source coefficient
	has the same endpoint verification.
	
	\subsection{The zero-entropy boundary and models in the subhomogeneous regions}
	\paragraph{Arbitrary positive entropy in $G_4$.}
	Given $H>0$, use again the hyperbolic space in \eqref{eq:busemann}.
	Set
	\[
	u=L+e^{-tb},\qquad L,t>0.
	\]
	The constant $L$ does not affect the gradient or the $m$-Laplace operator.
	By \eqref{eq:hyperboliccompute},
	\[
	-\Delta_m u
	=|\nabla u|^{m-1}[H-(m-1)t].
	\]
	In $G_4$, we have $p<0,q=m-1$.
	Since $u\ge L$, it follows that $u^p\le L^p$.
	First choose $L$ sufficiently large that $L^p<H$, and then take
	$0<t<(H-L^p)/(m-1)$.
	It follows that
	\[
	-\Delta_m u
	\ge L^p|\nabla u|^{m-1}
	\ge u^p|\nabla u|^{m-1}.
	\]
	This gives a nonconstant positive solution for every positive entropy $H$.
	Consequently, the $G_4$ assertion of Theorem \ref{thm:G46}
	cannot uniformly allow any fixed range of positive entropy.
	
	\paragraph{Every higher power in $G_6$.}
	Put
	\[
	d_0=m-1-q>0,\qquad k=d_0+1=m-q.
	\]
	Given any $\ell>k$ and $\lambda>0$, prescribe
	\begin{equation}\label{eq:G6powerchoices}
		\alpha=\frac{\ell-1}{d_0}>1,\qquad
		A_\infty(r)=e^{\lambda r^\ell},\qquad
		f(r)=C r^{-\alpha}
	\end{equation}
	near infinity, where $C>0$ will be chosen below.
	For a given $L>0$, the corresponding function near infinity is
	\begin{equation}\label{eq:G6powersolution}
		u(r)=L+\int_r^\infty Ct^{-\alpha}\dd t
		=L+\frac C{\alpha-1}r^{1-\alpha}.
	\end{equation}
	By \eqref{eq:radialflux},
	\begin{align}
		\frac{-\Delta_m u}{f^q}
		&=f^{d_0}\left[
		\lambda\ell r^{\ell-1}-(m-1)\frac\alpha r\right]\notag\\
		&=C^{d_0}r^{-\alpha d_0}
		\left[\lambda\ell r^{\ell-1}-(m-1)\frac\alpha r\right]\notag\\
		&=C^{d_0}\left[\lambda\ell-(m-1)\alpha r^{-\ell}\right].
		\label{eq:G6powercompute}
	\end{align}
	The last equality uses $\alpha d_0=\ell-1$.
	Since $u(r)^p\to L^p$, it suffices to choose
	\begin{equation}\label{eq:modelconstantchoice}
		\lambda\ell C^{d_0}>L^p,
	\end{equation}
	to obtain the strict inequality for all sufficiently large $r$.
	This choice is possible whether $p$ is positive, zero, or negative.
	
	To ensure the global weak-solution requirements, prescribe $f(r)=cr^a$
	near the origin, choosing
	\begin{equation}\label{eq:choosepole}
		0<a<1/d_0,\qquad n+aq>0.
	\end{equation}
	If $q\ge0$, the second condition is automatic.
	If $q<0$, it suffices to require additionally that $a<n/(-q)$.
	Thus both conditions can always be satisfied simultaneously.
	Connect this $f$ to the expression near infinity in \eqref{eq:G6powerchoices}
	through a positive smooth interpolation, and define $u=L+\int_r^\infty f$.
	Near the pole, $u$ has a finite positive limit,
	so the right-hand side of \eqref{eq:areainequality} is
	\[
	H(r)=O(r^{-ad_0})-\frac{(m-1)a}{r}.
	\]
	Since $ad_0<1$, we have
	$rH(r)\to-(m-1)a<n-1$.
	Thus the Euclidean area density $A_0$ near the pole satisfies $A_0'/A_0\ge H$.
	The corresponding condition near infinity follows from \eqref{eq:G6powercompute}.
	Lemma \ref{lem:glue} yields a global model and a positive $C^1$ weak solution.
	The construction need not produce a $C^2$ solution, but it meets the solution
	class used in this paper.
	
	We finally verify its volume growth.
	Near infinity, $A(r)=C_0e^{\lambda r^\ell}$, and
	\begin{equation}\label{eq:powervolumeasymptotic}
		V(R)\sim
		\frac{C_0}{\lambda\ell}
		R^{1-\ell}e^{\lambda R^\ell}.
	\end{equation}
	Indeed, the derivative of the expression on the right, divided by
	$C_0e^{\lambda R^\ell}$, tends to $1$.
	Since both comparison functions tend to infinity, a one-dimensional derivative
	comparison gives the stated asymptotic relation.
	Taking logarithms,
	\[
	\log V(R)=\lambda R^\ell+(1-\ell)\log R+O(1)
	\sim\lambda R^\ell.
	\]
	This proves the first assertion of Proposition \ref{prop:models}.
	
	\paragraph{An existence example with a logarithmic factor at the critical power.}
	Now fix $\beta>d_0$ and $\lambda>0$, and prescribe
	\begin{equation}\label{eq:G6logmodel}
		A_\infty(r)=\exp\bigl(\lambda r^k(\log r)^\beta\bigr),
		\qquad
		f(r)=\frac C{r(\log r)^{\beta/d_0}}
	\end{equation}
	near infinity.
	Since $\beta/d_0>1$,
	\[
	\int_r^\infty f(t)\dd t
	=\frac{C}{\beta/d_0-1}(\log r)^{1-\beta/d_0}<\infty.
	\]
	We may therefore take
	\begin{equation}\label{eq:G6logsolution}
		u(r)=L+\frac{C}{\beta/d_0-1}(\log r)^{1-\beta/d_0},
		\qquad L>0.
	\end{equation}
	A direct calculation gives
	\begin{align}
		\frac{A_\infty'}{A_\infty}
		&=\lambda r^{k-1}
		\left[k(\log r)^\beta+\beta(\log r)^{\beta-1}\right],
		\notag\\
		\frac{f'}f&=-\frac1r-\frac{\beta}{d_0r\log r},\qquad
		f^{d_0}=C^{d_0}r^{-d_0}(\log r)^{-\beta}.
		\label{eq:G6logderivatives}
	\end{align}
	Using $k-1=d_0$ and substituting into \eqref{eq:radialflux}, we obtain
	\begin{align}
		\frac{-\Delta_m u}{f^q}
		={}&\lambda C^{d_0}
		\left(k+\frac{\beta}{\log r}\right)\notag\\
		&-(m-1)C^{d_0}r^{-k}(\log r)^{-\beta}
		\left(1+\frac{\beta}{d_0\log r}\right)
		\longrightarrow\lambda kC^{d_0}.
		\label{eq:G6logcompute}
	\end{align}
	Choose $\lambda kC^{d_0}>L^p$.
	Since $u^p\to L^p$, the strict inequality holds near infinity.
	Using the same choice near the pole as in \eqref{eq:choosepole}
	and applying Lemma \ref{lem:glue}, we obtain a globally positive admissible weak solution.
	
	Write $\Phi(R)=\lambda R^k(\log R)^\beta$.
	Then $\Phi'(R)>0$ for sufficiently large $R$,
	and $\Phi''(R)/\Phi'(R)^2\to0$.
	Differentiating $e^{\Phi(R)}/\Phi'(R)$ gives
	\[
	\left(\frac{e^{\Phi(R)}}{\Phi'(R)}\right)'
	=e^{\Phi(R)}
	\left(1-\frac{\Phi''(R)}{\Phi'(R)^2}\right).
	\]
	Therefore $V(R)\sim C_0e^{\Phi(R)}/\Phi'(R)$, while
	$\log\Phi'(R)=O(\log R+\log\log R)=o(\Phi(R))$.
	It follows that
	$\log V(R)\sim\lambda R^k(\log R)^\beta$,
	which proves the second assertion of Proposition \ref{prop:models}.
	
	\subsection{A sharper classical benchmark at critical growth}
	\begin{proposition}[A classical stochastic-completeness benchmark]\label{prop:stochastic-slice}
		Suppose $m=2$, $q=0$, and $p<1$.
		If $M$ is stochastically complete, then there is no positive
		$C^2$ function satisfying $-\Delta u\ge u^p$.
		In particular, the curvature-free volume condition
		\[
		\int^\infty\frac{R}{\log(e+V(R))}\,\mathrm dR=\infty
		\]
		is sufficient. Consequently $V(R)\le C\exp(C_1R^2(\log R)^\beta)$,
		$\beta\le1$, is sufficient in this slice.
	\end{proposition}
	\begin{proof}
		We use the equivalence between stochastic completeness and the weak
		maximum principle at infinity \cite{PRS2005}, together with
		Grigor'yan's volume test \cite{Grig1999}.
		If $0\le p<1$, let $\theta=1-p$ and $w=-u^\theta$.
		The function $w$ is bounded above, and the chain rule gives
		\[
		\Delta w=-\theta u^{-p}\Delta u
		+\theta p u^{-p-1}|\nabla u|^2\ge\theta>0,
		\]
		contradicting that maximum principle.
		If $p<0$, take $w=-u$, write $w^*=\sup_Mw\le0$, and choose
		$\gamma<w^*$ with $\gamma<0$.
		On the nonempty upper level set $\{w>\gamma\}$, one has
		$0<u<-\gamma$, whence
		$\Delta w\ge u^p\ge(-\gamma)^p>0$, giving the same contradiction.
		The asserted logarithmic volume bound makes the volume-test integrand
		bounded below by a positive multiple of
		$1/[R(\log R)^\beta]$, whose integral diverges for $\beta\le1$.
	\end{proof}
	
	\begin{remark}
		This proposition is recalled as a classical benchmark, and is stated
		for $C^2$ supersolutions. It shows why the general condition
		$h_{m-q}(M)=0$ should not be described as a complete classification
		at critical growth. In the slice $m=2$, $q=0$, $p<1$, the model
		construction with $\log V(R)\sim\lambda R^2(\log R)^\beta$,
		$\beta>1$, matches the first logarithmic threshold above.
	\end{remark}
	
	For this comparison the existence models can also be made smooth at
	the pole. Keep the $m=2,q=0$ tails of \eqref{eq:G6logmodel}, with
	$\beta>1$, and choose $f_c=cr$ near zero. A positive smooth interpolation
	depending linearly on $c$ gives $u_c(0)=K_0+K_1c$, where $K_0,K_1>0$
	are independent of $c$, and leaves the tail unchanged.
	In the Euclidean region $-\Delta u_c=nc$.
	If $0\le p<1$, choose $c$ large enough that
	$nc\ge(K_0+K_1c)^p$; then $u_c\le u_c(0)$ verifies the source inequality.
	If $p<0$, choose $nc\ge L^p$ and use $u_c\ge L$.
	Lemma~\ref{lem:glue} therefore applies with this explicit pole check.
	Near the pole $u_c=u_c(0)-cr^2/2$, so both the metric and the
	supersolution are smooth everywhere. Thus the logarithmic threshold
	in Proposition~\ref{prop:stochastic-slice} is matched by examples
	within its classical solution class.
	
	\subsection{Relations among the conclusions and unresolved endpoints}
	The proofs retain three different kinds of information.
	In $G_1,G_3$, annular mass estimates for a fixed density yield integral conditions,
	from which iterated logarithms in the volume bounds are read directly through
	a one-dimensional integral.
	In $G_4,G_6$, a negative-power parameter increasing with the radius produces an
	exponentially small factor, with its scale determined by
	$a^{1/h}/R^{(m-q)/h}$.
	In $G_5$, the logarithmic transformation produces an exact one-dimensional
	minimum that is compared directly with the decay rate of the exponential weight.
	These mechanisms reflect different homogeneity structures.
	A parameter choice suitable for one region therefore cannot be transferred
	mechanically to another.
	
	For the first three regions, slower factors can still be inserted after the
	first critical logarithm.
	This is consistent with the criticality of the first logarithmic exponent itself.
	For instance, every fixed $\varepsilon>0$ satisfies
	\[
	(\log\log R)^B=o((\log R)^\varepsilon).
	\]
	Thus an existence model at the volume scale $R^D(\log R)^{B+\varepsilon}$
	does not automatically rule out a nonexistence result at the smaller scale
	\[
	R^D(\log R)^B(\log\log R)^B.
	\]
	The integral criteria used here keep these distinct levels of comparison separate.
	
	For the models, every finite logarithmic level in $G_1,G_2,G_3$ is
	sharp by Theorem~\ref{thm:logsharp}. The strict entropy threshold in $G_5$ is verified by
	endpoint examples, whereas the results in $G_6$ match from the two sides
	only at the level of the power scale.
	The restriction $\beta>d_0$ in \eqref{eq:G6logmodel} comes from
	$\int^\infty f<\infty$ and is a requirement of the present bounded radial construction.
	It must not be regarded as a necessary existence condition for all models
	or all positive solutions.
	Accordingly, no such nonexistence assertion for all $G_6$ parameters
	and $\beta\le d_0$ is proved here. The classical slice just discussed
	has additional information from stochastic completeness.
	
	Finally, this paper does not transfer the smooth-manifold chain rule directly to graphs.
	Discrete product identities, the choice of distance, and the structure of
	second-order operators require separate treatment.
	The assumptions used in \cite{HS2023,DLMQS2027} reflect these differences.
	The proofs here close without curvature assumptions or distance-Laplacian bounds
	because the continuous chain rule in
	\eqref{eq:basicnegative}--\eqref{eq:basicordinary}
	is combined with first-order cutoff estimates.
	To extend the fixed-density annular summation argument to other operators
	or geometric settings, one should first verify the corresponding local estimates,
	rather than merely replace the volume notation.

	\section{Weighted and Sobolev extensions}\label{sec:weighted}
	\subsection{Smooth positive measures}
	Let $\Phi\in C^\infty(M)$ and equip the same complete metric with
	the measure $\dd\nu=e^{-\Phi}\dd\mu$. Define
	\[
	\Delta_{m,\Phi}u
	=e^\Phi\operatorname{div}
	(e^{-\Phi}|\nabla u|^{m-2}\nabla u),\qquad
	V_\nu(R)=\nu(B(o,R)).
	\]
	The weak inequality is defined by
	\begin{equation}\label{eq:weightedweak}
		\int_M |\nabla u|^{m-2}\langle\nabla u,\nabla\psi\rangle\dd\nu
		\ge\int_Mu^p|\nabla u|^q\psi\dd\nu,
		\quad 0\le\psi\in\operatorname{Lip}_c(M).
	\end{equation}
	The positive $C^1$ solution and singular-power conventions are those
	of Definition~\ref{def:solution}, with local integrability relative
	to $\nu$.
	
	\begin{theorem}[Weighted extension without curvature conditions]\label{thm:weighted}
		Theorems~\ref{thm:G1}, \ref{thm:G3}, \ref{thm:G46}, and
		\ref{thm:G5}, Proposition~\ref{prop:G2}, and their nonexistence
		corollaries hold for \eqref{eq:weightedweak} after replacing $V$ by
		$V_\nu$. The entropy constant $c_*(m,p)$, and its source-coefficient
		version \eqref{eq:cstar-lambda}, are unchanged.
		No restriction on $\Phi$, its gradient, its Hessian, or any
		Bakry--\'Emery curvature tensor is required.
	\end{theorem}
	
	\begin{proof}
		On every compact set, the smooth density $e^{-\Phi}$ is bounded above
		and below by positive constants. The local test approximation and
		admissibility arguments of Section~\ref{sec:prelim} therefore remain
		valid. Substitution of $u^{-a}\phi^b$ into \eqref{eq:weightedweak}
		gives exactly \eqref{eq:basicnegative}, with $\dd\nu$ as measure;
		the ordinary test gives \eqref{eq:basicordinary} in the same way.
		H\"older and Young inequalities hold for $\nu$, and every geometric
		estimate used thereafter is either $|\nabla r|\le1$ or
		$\nu(\mathcal A_R)\le V_\nu(4R)$.
		The mass recurrence requires only monotonicity and finiteness of
		$V_\nu(R)$, both of which hold.
		
		For the subhomogeneous argument, the fixed compact set on which
		$|\nabla u|>0$ has positive $\nu$ measure, so normalization and
		the large-parameter estimate are unchanged. On the homogeneous line,
		testing with $u^{1-m}\psi$ gives \eqref{eq:logweak} relative to $\nu$.
		The minimum of $(m-1)t+t^{-p}$ is unchanged, and the exponentially
		weighted cutoff has the same first derivative estimate. Hence the
		entropy proof and its exact constant carry over.
		At no step is the density differentiated or a curvature estimate used.
	\end{proof}
	
	The endpoint and logarithmic examples already belong to this class
	with $\Phi=0$, so the corresponding universal sharpness statements
	remain valid for the larger class of smooth weighted manifolds.
	The extension does not assert a characterization of existence from
	weighted volume on every fixed weighted geometry.
	
	\subsection{Locally bounded Sobolev supersolutions}
	\begin{proposition}\label{prop:sobolev}
		The nonexistence and Liouville conclusions of Theorems~\ref{thm:G1},
		\ref{thm:G3}, \ref{thm:G46}, \ref{thm:G5}, and
		Proposition~\ref{prop:G2}, including their volume corollaries, remain valid for
		$u\in W^{1,m}_{\loc}(M)$ such that
		\[
		0<\operatorname*{ess\,inf}_K u
		\le \operatorname*{ess\,sup}_K u<\infty
		\]
		for every relatively compact set $K$, provided
		$u^p|\nabla u|^q\in L^1_{\loc}(M)$ and the weak inequality holds for
		nonnegative compactly supported Lipschitz tests.
		This statement also holds for the measure $\nu$ above.
		Constancy in this class is understood almost everywhere.
	\end{proposition}
	
	\begin{proof}
		A bounded nonnegative compactly supported $W^{1,m}$ function can be
		approximated by nonnegative Lipschitz functions supported in one fixed
		larger compact set, uniformly bounded and convergent strongly in
		$W^{1,m}$ and almost everywhere. To obtain such a sequence, use local
		smoothing and a partition of unity, and then apply a Lipschitz
		truncation above a fixed bound larger than the essential supremum.
		The flux is locally in $L^{m/(m-1)}$, so H\"older's inequality passes
		the left-hand side to the limit. Uniform boundedness and dominated
		convergence against the locally integrable source pass the right-hand
		side to the limit.
		The weak formulation therefore admits these bounded Sobolev tests.
		
		On the local essential range of $u$, every power and logarithm used
		in the proof has bounded derivative. In particular,
		$u^{-a}\phi^b$ and $u^{1-m}\psi$ are bounded admissible Sobolev tests.
		All factorizations, energy estimates, and the fixed-density mass
		recurrence follow unchanged.
		In the subhomogeneous proof, a nonconstant Sobolev function has
		$|\nabla u|\ne0$ on a set of positive measure in some relatively
		compact set. For some integer $j$, the measurable subset
		$E=\{1/j\le |\nabla u|\le j\}$ in that set has positive measure.
		The local bounds on $u$ provide the same finite positive fixed mass
		used in \eqref{eq:fixedmass}. Thus continuity of the gradient is not
		needed there. Finally, vanishing of the weak gradient on a connected
		manifold implies constancy almost everywhere.
		Smooth positive densities are locally comparable to $\mu$, so the
		same approximation and argument apply to $\nu$.
	\end{proof}
	
	\subsection{Basepoints, finite volume, and positive source factors}
	The choice of basepoint is immaterial for all the criteria.
	If $d(o,o')=c$, the inclusions
	$B(o',R-c)\subset B(o,R)\subset B(o',R+c)$ compare the volumes.
	For large $R$, the powers of $R$ and $R\pm c$ are comparable, so
	a change of variable transfers divergence of each integral in
	\eqref{eq:unified} between the two basepoints.
	The same inclusions and $(R\pm c)^s/R^s\to1$ show that the lower
	rates $h_s$ are independent of the basepoint.
	These observations use no doubling assumption.
	
	\begin{corollary}[Finite weighted volume]\label{cor:finitevolume}
		If $\nu(M)<\infty$, then the Liouville property holds for
		$-\Delta_{m,\Phi}u\ge u^p|\nabla u|^q$ for every $p,q\in\mathbb R$
		in either admissible class used above.
	\end{corollary}
	\begin{proof}
		In $G_1,G_2,G_3$ the integrals \eqref{eq:unified} diverge because
		$D_i,B_i>0$ and $V_\nu$ is bounded above.
		In the other three regions all $h_s$ vanish, which implies the
		stated conclusions by Theorems~\ref{thm:G46} and \ref{thm:G5}.
	\end{proof}
	
	More generally, if the source is $a(x)u^p|\nabla u|^q$ with a
	measurable coefficient $a(x)\ge\lambda>0$, every conclusion for the
	constant coefficient $\lambda$ applies by direct comparison in the
	weak formulation, provided the source is locally integrable.
	The only resulting change in the stated volume criteria is the
	homogeneous-line constant \eqref{eq:cstar-lambda}.
	
	\begingroup
	\interlinepenalty=10000

	\endgroup
\end{document}